\documentclass[12pt,reqno]{amsart}
\usepackage{amsthm,amsfonts,amssymb,euscript}

\newcommand{\bea}{\begin{eqnarray}}
\newcommand{\eea}{\end{eqnarray}}
\def\beaa{\begin{eqnarray*}}
\def\eeaa{\end{eqnarray*}}
\def\ba{\begin{array}}
\def\ea{\end{array}}
\def\be#1{\begin{equation} \label{#1}}
\def \eeq{\end{equation}}

\def\be{{\beta}}
\def\ga{\gamma}

\def\eps{\epsilon}

\def\si{\sigma}

\def\th{\theta}

\def\al{\alpha}

\def\GG{{\mathcal G}}

\def\Z{{\bf Z}}

\def\G{{\mathbb G}}
\def\Z{{\mathbb Z}}

\def\f12{{\frac 1 2}}
\def\bH{{\widetilde{H}}}

\def\th{\theta}

\def\va{\vartheta}

\def\f{\widetilde{f}}

\def\um{\underline{u}}
\def\vm{\underline{v}}

\newtheorem{theorem}{Theorem}[section]
\newtheorem{lemma}[theorem]{Lemma}
\newtheorem{proposition}[theorem]{Proposition}

\newtheorem{definition}[theorem]{Definition}
\newtheorem{remark}[theorem]{Remark}

\numberwithin{equation}{section}

\begin{document}
\title[Averages along polynomial sequences in nilpotent groups]{Averages along polynomial sequences in discrete nilpotent groups: singular Radon transforms}
\author{Alexandru D. Ionescu}
\address{Princeton University}
\email{aionescu@math.princeton.edu}
\author{Akos Magyar}
\address{University of British Columbia}
\email{magyar@math.ubc.ca}
\author{Stephen Wainger}
\address{University of Wisconsin--Madison}
\email{wainger@math.wisc.edu}

\thanks{The first author was partially supported by a Packard Fellowship and NSF grant DMS-1065710. The second author was partially supported by NSERC grant 22R44824.}

\begin{abstract}

We consider a class of operators defined by taking averages along polynomial sequences in discrete nilpotent groups. As in the continuous case, one can consider discrete maximal Radon transforms, which have applications to pointwise ergodic theorems, and discrete singular Radon transforms. In this paper we prove $L^2$ boundedness of discrete singular Radon transforms along general polynomial sequences in discrete nilpotent groups of step 2.
\end{abstract}

\maketitle

\tableofcontents

\section{Introduction}\label{intro}

A class of interesting problems arises in studying averages of functions along polynomial sequences in discrete nilpotent groups. More precisely, assume $\mathbb{G}$ is a discrete nilpotent group of step $d\geq 1$ and $A:\mathbb{Z}\to\mathbb{G}$ is a polynomial sequence (see Definition \ref{MainDef} below), and consider the following problems:{\footnote{One can also state similar problems in the case of $L^q$ functions, $q>1$, or for multi-dimensional polynomial sequences $A:\mathbb{Z}^{k}\to\mathbb{G}$, $k\geq 1$.}}
\medskip

{\bf{Problem 1.}} ($L^2$ boundedness of maximal Radon transforms) Assume $f:\mathbb{G}\to\mathbb{C}$ is a function and let
\begin{equation*}
\mathcal{M}f(g)=\sup_{N\geq 0}\frac{1}{2N+1}\sum_{|n|\leq N}|f(A^{-1}(n)\cdot g)|,\qquad g\in\mathbb{G}.
\end{equation*}
Then
\begin{equation*}
\|\mathcal{M}f\|_{L^2(\mathbb{G})}\lesssim \|f\|_{L^2(\mathbb{G})}.
\end{equation*}
\medskip

{\bf{Problem 2.}} ($L^2$ pointwise ergodic theorems) Assume $\mathbb{G}$ acts by measure-preserving transformations on a probability space $X$, $f\in L^2(X)$, and let
\begin{equation*}
A_Nf(x)=\frac{1}{2N+1}\sum_{|n|\leq N}f(A^{-1}(n)\cdot x),\qquad x\in X.
\end{equation*}
Then the sequence $A_Nf$ converges almost everywhere in $X$ as $N\to\infty$.
\medskip

{\bf{Problem 3.}} ($L^2$ boundedness of singular Radon transforms) Assume $K:\mathbb{R}\to\mathbb{R}$ is a Calderon--Zygmund kernel (see \eqref{CalZyg}), $f:\mathbb{G}\to\mathbb{C}$ is a (compactly supported) function, and let
\begin{equation*}
Hf(g)=\sum_{n\in\mathbb{Z}} K(n)f(A^{-1}(n)\cdot g),\qquad g\in\mathbb{G}.
\end{equation*}
Then
\begin{equation*}
\|Hf\|_{L^2(\mathbb{G})}\lesssim \|f\|_{L^2(\mathbb{G})}.
\end{equation*}
\medskip

The maximal Radon transform and the singular Radon transform can be thought of as discrete analogues of the continuous Radon transforms, which are averages along suitable curves or surfaces in Euclidean spaces. The theory of continuous Radon transforms has been extensively studied and is very well understood (including $L^q$, $q>1$, estimates and multidimensional averages), see for example \cite{Ch}, \cite{RiSt}, \cite{ChNaStWa}.

In the discrete setting, the three questions raised above have been answered in the affirmative in the commutative case $\G=\Z^d$.{\footnote{The linear case $\mathbb{G}=\mathbb{Z}$, $A(n)=n$, is, of course, well-known.}} The maximal function estimate and the pointwise ergodic theorem were proved by Bourgain \cite{Bo1}, \cite{Bo2}, \cite{Bo3}, also in the case of $L^q$ functions, $q>1$. $L^2$ estimates for singular Radon transforms were obtained in \cite{ArOs}, the $L^q$ boundedness was established in \cite{StWa1} for $3/2<q<3$ and were extended for all $q>1$ in \cite{IoWa}. Closely related fractional integral operators were treated in \cite{Ob}, \cite{StWa3}, \cite{Pi1}, \cite{Pi2}.

Only partial results are available, however, in the case non-commutative discrete nilpotents groups, even in the case of step 2 nilpotent groups. A general feature of the partial results obtained in the non-commutative setting, see \cite{IoMaStWa}, \cite{MaStWa}, \cite{StWa4}, is that the averages are taken over surfaces transversal to the center of the group, such that the "non-linear" part of the polynomial map is contained in the center. The point is that for such special polynomial sequences one can still use the Fourier transform in the central variables to analyze the operators. 

However, it appears that one needs to proceed in an entirely different way in the case of general polynomial maps, when the Fourier transform method is not available. The present work is the first attempt to treat discrete Radon transforms along general polynomial sequences in the non-commutative nilpotent settings. More precisely, we will discuss the easier Problem 3 in the case of discrete nilpotent groups of step $2$.

Finally let us remark that the $L^2$ ergodic theorems of Bergelson and Leibman \cite{BeLe} indicate that nilpotent groups provide the most general settings to which the results of Bourgain might extend. Indeed, they have shown that averages of measure preserving transformations generating a nilpotent group converge in the mean along any polynomial sequence, however this does not hold for transformations generating a solvable group.

To describe our settings in detail, recall that a polynomial sequence on a nilpotent group $\G$ is a map $A:\Z\to\G$, such that $D^k A(n)=1$ for all $n$ for some fixed $k$, where $D^k$ is the $k$-fold iterate of the differencing operator $D$ defined by $DA(n)=A(n)^{-1}A(n+1)$. It is known, see \cite{Le} that $A$ is a polynomial sequence if and only if $A(n)=g_1^{p_1(n)}\ldots g_t^{p_t(n)}$ for all $n$, where $g_1,\ldots,g_t$ are elements of $\G$ and $p_1,\ldots,p_t$ are integral polynomials. In particular the image of the map $A$ is contained in a finitely generated subgroup of $\G$, thus without the loss of generality we will assume that $\G$ is finitely generated and hence countable. We will also assume that $\G$ is torsion free and then, by a result of Malcev \cite{Mal}, the group $\G$ can be embedded as a discrete, co-compact subgroup of a (connected and simply connected) nilpotent Lie group $\G^\sharp$. This motivates the following:

\begin{definition}\label{MainDef}
Given $d\geq 1$, a group $\G$ will be called a discrete nilpotent group of step $d$ if $\mathbb{G}$ is isomorphic to a discrete, co-compact subgroup of a (connected and simply connected) nilpotent Lie group $\G^\sharp$ of step $d$. 

Given a group $\G$, a sequence $A:\Z\to\G$ will be called a polynomial sequence if $A(0)=1$ and $D^{k_0}A\equiv 1$ for some $k_0\geq 1$, where, by definition,
\begin{equation*}
D^0A(n)=A(n),\qquad D^{k+1}A(n)=D^kA(n)^{-1}D^kA(n+1),\qquad n\in\mathbb{Z}.
\end{equation*}
\end{definition}

In this paper we consider only the easier problem of $L^2$ boundedness of the discrete singular Radon transforms. To formulate our main result, let $K:\mathbb{R}\to\mathbb{R}$ be a Calderon--Zygmund kernel, i.e. a $C^1$ function satisfying
\begin{equation}\label{CalZyg}
\sup_{t\in\mathbb{R}}[(1+|t|)|K(t)|+(1+|t|)^2|K'(t)|]\leq 1,\qquad\sup_{N\geq 0}\Big|\int_{-N}^NK(t)\,dt\Big|\leq 1.
\end{equation}
The main theorem we prove in this paper is the following:

\begin{theorem}\label{Main1}
Assume $\G$ is a discrete nilpotent group of step $2$, $K$ is a Calderon--Zygmund kernel, and $A:\mathbb{Z}\to\G$ is a polynomial sequence. For any (compactly supported) function $f:\G\to\mathbb{C}$ let
\begin{equation*}
(Hf)(g)=\sum_{n\in\Z}K(n)f(A^{-1}(n)\cdot g),\qquad g\in\G.
\end{equation*}
Then
\begin{equation*}
\|Hf\|_{L^2(\G)}\lesssim\|f\|_{L^2(\G)}.
\end{equation*}
\end{theorem}

We describe now some of the main ideas in the proof of the theorem. We use first a transference principle to reduce matters to proving the theorem in a certain "universal" case. More precisely, it will suffice to consider singular Radon transforms on the groups $\mathbb{G}_0=\mathbb{G}_0(d)$ defined in section \ref{transference}, and for explicit polynomial sequences $A_0:\mathbb{Z}\to\mathbb{G}_0$, see Theorem \ref{Main2}. This reduction simplifies the overall picture and allows us to work in good systems of coordinates, which are well adapted to the natural homogeneities induced by the polynomial $A_0$. However, the main problem, namely the lack of a good Fourier transform on the group $\mathbb{G}_0$ compatible with the structure of our convolution operators, remains even in this special setting.

A natural approach is to attempt to prove the theorem using the Cotlar--Stein lemma. More precisely, we may assume that
\begin{equation*}
K=\sum_{j=1}^\infty K_j,\qquad \int_{\mathbb{R}}K_j(t)\,dt=0,\qquad 2^j|K_j(t)|+2^{2j}|K'_j(t)|\leq \mathbf{1}_{[-2^{j+3},2^{j+3}]}(t),
\end{equation*}
and consider the dyadic averages
\begin{equation*}
H_j(f)(g)=\sum_{n\in\mathbb{Z}}K_j(n)f(A_0(n)^{-1}\cdot g),\qquad g\in\mathbb{G}_0.
\end{equation*}
To apply the Cotlar--Stein lemma, we would have to prove an inequality of the form
\begin{equation}\label{introd1}
\|H_kH_j^\ast\|_{L^2\to L^2}+\|H_k^\ast H_j\|_{L^2\to L^2}\lesssim 2^{-\delta'(j-k)}
\end{equation}
for some $\delta'>0$, and for any $k\leq j\in\{1,2,\ldots\}$. This is equivalent to proving that
\begin{equation}\label{introd2}
\|H_k(H_j^\ast H_j)^r\|_{L^2\to L^2}+\|H_k^\ast (H_jH_j^\ast)^r\|_{L^2\to L^2}\lesssim 2^{-\delta(j-k)}
\end{equation}
for some $\delta>0$, $r\in\{1,2,\ldots\}$, and for any $k\leq j\in\{1,2,\ldots\}$.

The advantage of proving \eqref{introd2} instead of \eqref{introd1} is that the operators $(H_j^\ast H_j)^r$ and $(H_jH_j^\ast)^r$ are more regular than the operators $H_j$, provided that $r\geq r(d)$ is sufficiently large. The kernels of these operators can be described precisely, see Proposition \ref{majarcs}. Up to negligible errors, these operators are essentially sums of more standard oscillatory singular operators on the group $\mathbb{G}_0$, given by kernels of the form\footnote{The proof of Proposition \ref{majarcs}, which includes this description, relies on the complicated oscillatory sum estimates in Proposition \ref{minarcs}. Having an elementary, essentially self-contained proof of these estimates is the main reason for working on step 2 groups, instead of the general case.}
\begin{equation}\label{introd9}
h\to\sum_{a/q}S^{(r)}(a/q)e^{2\pi ih\cdot a/q}K_j^{(r)}(h).
\end{equation}
The sum is taken over suitable "irreducible fractions" $a/q$, the coefficients $S^{(r)}(a/q)$ have sufficiently fast decay decay as $q\to\infty$ (provided that $r$ is sufficiently large), and $K_J^{(r)}$ is (almost) a standard singular integral kernel adapted to the canonical non-isotropic balls on the underlying Lie group $\mathbb{G}_0^\#$. This representation can be used to prove that
\begin{equation*}
\|H_k(H_j^\ast H_j)^r\|_{L^2\to L^2}+\|H_k^\ast (H_jH_j^\ast)^r\|_{L^2\to L^2}\lesssim 2^{-\delta(j-k)}+2^{-\delta k},\qquad\delta>0,\,k\leq j\in\{1,2,\ldots\},
\end{equation*}
see Lemma \ref{separated2}, and, as a consequence,
\begin{equation}\label{introd7}
\|H_kH_j^\ast\|_{L^2\to L^2}+\|H_k^\ast H_j\|_{L^2\to L^2}\lesssim 2^{-\delta'(j-k)}+2^{-\delta' k},\qquad\delta'>0,\,k\leq j\in\{1,2,\ldots\}.
\end{equation}
Unfortunately this last bound is weaker than the desired bound \eqref{introd1}, and the additional factor $2^{-\delta' k}$ cannot be removed. As a consequence, the Cotlar--Stein lemma can be used to prove the weaker bound
\begin{equation*}
\big\|\sum_{j\in[J,2J]}H_j\big\|_{L^2\to L^2}\lesssim 1,\qquad\text{ uniformly in }J,
\end{equation*}
but is not suitable to control the entire sum over $j$.

To estimate the entire sum we need an additional almost-orthogonality lemma, which we prove in section \ref{orthog}. This lemma appears to be new and might be of independent interest. In its simplest form, it says that if $S_1,\ldots S_K$ are bounded linear operators on a Hilbert space $H$ satisfying, for any $m=1,\ldots,K$,
\begin{equation}\label{introd6}
\begin{split}
&\sup_{m\in\{1,\ldots,K\}}\|S_m\|\leq 1,\\
&\sup_{i_m,\ldots,i_K\in \{0,1\}}\|S^{\ast}_{m,i_m}[(S_{m+1,i_{m+1}}S^\ast_{m+1,i_{m+1}})^{p_0}+\ldots+(S_{K,i_K}S^\ast_{K,i_K})^{p_0}]\|\leq A2^{-\delta_0m},\\
&\sup_{i_m,\ldots,i_K\in \{0,1\}}\|S_{m,i_m}[(S^\ast_{m+1,i_{m+1}}S_{m+1,i_{m+1}})^{p_0}+\ldots+(S^\ast_{K,i_K}S_{K,i_K})^{p_0}]\|\leq A2^{-\delta_0m},
\end{split}
\end{equation}
for some $\delta_0>0$, some dyadic number $p_0$, and some constant $A$, then
\begin{equation*}
\|S_1+\ldots+S_K\|\leq C(\delta_0,A,p_0).
\end{equation*}
The notation in \eqref{introd6} is $S_{m,0}=S_m$ and $S_{m,1}=0$.

We apply this almost-orthogonality lemma with
\begin{equation*}
S_m=\sum_{j\in [(1-\kappa)J_m,J_m]}H_j,
\end{equation*}
where $\kappa>0$ is a sufficiently small constant and $J_1,J_2,\ldots$ is a rapidly increasing sequence, $J_{m+1}\geq 2J_m$. The inequality in the first line of \eqref{introd6} is a consequence of the Cotlar--Stein lemma and \eqref{introd7}. We prove the remaining inequalities in \eqref{introd6} in two steps: in Lemma \ref{separated3} we prove the uniform bounds
\begin{equation*}
\big\|(S_m^\ast S_m)^r+\ldots+(S_n^\ast S_n)^r\big\|_{L^2\to L^2}+\big\|(S_mS_m^\ast)^r+\ldots+(S_n S_n^\ast)^r\big\|_{L^2\to L^2}\lesssim 1,
\end{equation*}
for any $m\leq n\in\{1,2,\ldots\}$. For this we establish formulas similar to \eqref{introd9} for the kernels of the operators $(S_k^\ast S_k)^r$ and $(S_k^\ast S_k)^r$. Then we show in Lemma \ref{separated5} that left composition with the operator $S_{m-1}$ (or $S^\ast_{m-1}$ respectively) contributes an additional factor of $2^{-\delta m}$, $\delta>0$, thereby proving the desired bounds in \eqref{introd6}.

The rest of the paper is organized as follows. In section \ref{transference} we use a transference argument to reduce the general case in Theorem \ref{Main1} ( corresponding to a general group $\G$ and a general sequence $A:\mathbb{Z}\to\G$) to a "universal" case (corresponding to a particular group $\mathbb{G}_0$ and a particular sequence $A_0:\mathbb{Z}\to\mathbb{G}_0$).

In section \ref{oscl2} we define the operators $H_j$ (the dyadic pieces of our singular Radon transforms), and describe the operators $H_{j_1}^\ast H_{k_1}\ldots H_{j_r}^\ast H_{k_r}$ and $H_{j_1} H_{k_1}^\ast\ldots H_{j_r} H_{k_r}^\ast$,  for integers $j_1,k_1,\ldots j_r,k_r\in[J(1-\kappa),J]$. For $r\geq r(d)$ large enough we prove in Proposition \ref{majarcs} that the kernels of these operators are sums of more standard oscillatory singular integral kernels, similar to \eqref{introd9} (arising from "major arcs"), and negligible errors (arising from "minor arcs"). The bounds on these error terms rely on Proposition \ref{minarcs} and are delicate in our situation, due to the complicated structure of the polynomials that arise as a result of multiplication in the group $\mathbb{G}_0$.

Section \ref{proofthm} contains the proof of Theorem \ref{Main2}, i.e. the proof of the bounds in \eqref{introd6}, along the line described above.

In section \ref{oscil} we prove estimates for trigonometric sums and integrals, using a variant of the Weyl method developed by Davenport \cite{Da1} and Birch \cite{Bi}. These estimates are used at several places, for example to control the contributions of the "minor arcs" and to estimate the coefficients $S^{(r)}(a/q)$ in \eqref{introd9}. For the sake of completeness we provide all the details needed in the proof.

Finally, in section \ref{orthog} we state and prove a suitable version of the additional orthogonality lemma described in \eqref{introd6}.

{\bf{Acknowledgement:}} We would like to express our deep gratitude to Elias Stein, for his guidance and friendship throughout the years. 

\section{A transference argument}\label{transference}

Let $\G^\#$ be a step 2 (connected and simply connected) nilpotent Lie group and let $\GG$ denote its Lie algebra. Choose a basis $\mathcal{X}=\{X_1,\ldots,X_{d_1},Y_1,\ldots,Y_{d_2}\}$ of the Lie algebra $\GG$ such that $\mathbb{R}-span\,\{Y_1,\ldots,Y_{d_2}\}=[\GG,\GG]$, the commutator subalgebra of $\GG$. Note that this is a special case of a so-called strong Malcev basis passing through the lower central series $\GG\geq [\GG,\GG]\geq\{0\}$ (see \cite{CoGr}, Sec. 1.2). Associated to such a basis one defines coordinates on $\G^\#$ via the diffeomorphism $\phi:\mathbb{R}^d\to\G^\#$ defined by
\[\phi(x_1,\ldots,x_{d_1},y_1,\ldots,y_{d_2})=\exp(x_1X_1)\ldots\exp(x_{d_1}X_{d_1})
\exp(y_1Y_1)\ldots\exp(y_{d_2}Y_{d_2}).\]
Such coordinates associated to a Malcev basis are called exponential coordinates of the second kind. In these coordinates we have that
\begin{equation}\label{mult}
\G^\#=\{(x,y)\in\mathbb{R}^{d_1}\times\mathbb{R}^{d_2}:(x,y)\cdot (x',y')=(x+x',y+y'+R(x,x')\},
\end{equation}
where $R:\mathbb{R}^{d_1}\times\mathbb{R}^{d_1}\to\mathbb{R}^{d_2}$ is a bilinear form. This follows easily from facts that $\exp(X)\cdot\exp(Y)=\exp(X+Y+\frac{1}{2}[X,Y])$ which implies that
\[\exp(x_iX_i)\exp(x_j'X_j)=\exp(x_j'X_j)\exp(x_iX_i)\exp(x_ix_j'[X_i,X_j]),\] and $[X_i,X_j]=\sum_{l=1}^{d_2} c^l_{ij}Y_l$.

If $\G\leq\G^\#$ is a discrete co-compact subgroup, then one can choose such a basis $\mathcal{X}=\{X_1,\ldots,Y_{d_2}\}$ so that
\[\G=\phi(\Z^d)=\exp(\Z X_1)\ldots\exp(\Z X_{d_1})\exp(\Z Y_1)\ldots\exp(\Z Y_{d_2}),\] see \cite{CoGr} Thm. 5.1.6 and Prop. 5.3.2. Thus the discrete subgroup $\G$ is identified with the integer lattice $\Z^d=\Z^{d_1}\times\Z^{d_2}$.

If $A:\mathbb{Z}\to\G$ is a polynomial sequence ($A(0)=1$), then it is not hard to see that in these coordinates it takes the form
\begin{equation*}
A(n)=(x_1(n),\ldots,x_{d_1}(n),y_1(n),\ldots,y_{d_2}(n)),\qquad A(0)=0,
\end{equation*}
where $x_{l_1},y_{l_2}$ are integral polynomials. Indeed, writing
\[DA(n)=(Dx_1(n),\ldots,Dx_{d_1}(n),Dy_1(n),\ldots,Dy_{d_2}(n)),\] we have form (\ref{mult}) that $Dx_i(n)=x_i(n+1)-x_i(n)$ and $Dy_l(n)=y_l(n+1)-y_l(n)-R'_l(n)$ where $R'_l(n)$ is a polynomial expression of $x_1(n),\ldots,x_{d_1}(n),x_1(n+1),\ldots,x_{d_1}(n+1)$. Since $D^kx_i(n)$ is identically zero it follows that $x_i(n)$ is a polynomial of degree at most $k$, and then the vanishing of $D^ky_l(n)$ implies that $y_l(n)$ must be polynomial as well. Alternatively this fact can be easily derived from the characterization of polynomial sequences by Leibman \cite{Le} mentioned in the introduction. We will denote by $d_3$ the maximum of the degrees of the polynomials $x_i(n)$ and $y_l(n)$.

It will be useful to consider the polynomial map $A:\Z\to\G$ as a map $A:\Z\to\G^\#$, and the associated singular Radon transform acting on $L^2(\G^\#)$, defined by
\[(\bH f)(g)=\sum_{n\in\Z}K(n)f(A^{-1}(n)\cdot g),\qquad g\in\G^\#.\] In this settings our main result takes the form
\begin{theorem}\label{Main2.1}
Assume $\G^\#$ is a (connected and simply connected) nilpotent Lie group of step $2$, $K$ is a Calderon--Zygmund kernel, and $A:\mathbb{Z}\to\G^\#$ is a polynomial sequence. For any (continuous compactly supported) function $f:\G^\#\to\mathbb{C}$, we have
\begin{equation*}
\|\bH f\|_{L^2(\G^\#)}\lesssim\|f\|_{L^2(\G^\#)}.
\end{equation*}
\end{theorem}

We will show below that
\[\|\bH\|_{L^2(\G^\#)\to L^2(\G^\#)} = \|H\|_{L^2(\G)\to L^2(\G)},\]
hence Theorem \ref{Main2.1} and Theorem \ref{Main1} are equivalent. To see this let $\mathbf{S}_d=\phi([0,1)^d)$ where $\phi:\mathbb{R}^d\to\G^\#$ is the coordinate map defined above. From the multiplication structure given in (\ref{mult}) it is easy to see that $\mathbf{S}_d$ is a fundamental domain for $\G$, that is every element $g\in\G^\#$ can be written uniquely as $g=\ga\cdot s$ with $\ga\in\G$ and $s\in\mathbf{S}_d$. Moreover the map $\tilde{\phi}=\pi\circ\phi$ ($\pi$ being the natural projection from $\G^\#$ to $\G\backslash \G^\#$) maps the Lebesgue measure on $[0,1)^d$ to the normalized $\G^\#$-invariant measure on $\G\backslash \G^\#$. For a given function $f:\G\to\mathbb{C}$ let $f^\#:\G^\#\to\mathbb{C}$ be such that $f^\# (\ga\cdot s)=f(\ga)$ for all $\ga\in\G$ and $s\in\mathbf{S}_d$. Then
\[\|f^\#\|_{L^2(\G^\#)}^2 =\int_{\G^\#} |f^\# (g)|^2\,dg =\sum_{\ga\in\G} \int_{\mathbf{S}_d} |f^\#(\ga\cdot s)|^2\,ds = \sum_{\ga\in\G} |f(\ga)|^2 =\|f\|_{L^2(\G)}^2.\]
Also
\[\bH f^\# (\ga\cdot s)=\sum_{n\in\Z}K(n) f^\# (A(n)^{-1}\cdot\ga\cdot s)=\sum_{n\in\Z}K(n) f (A(n)^{-1}\cdot\ga)=Hf(\ga),\] thus $\bH f^\#=(H f)^\#$ and hence the operators $\bH$ and $H$ have the same norm.

The advantage of Theorem \ref{Main2.1} is that it is easier to reduce it to a certain universal case. For integers $d\geq 1$ we define
\begin{equation*}
Y_d=\{(l_1,l_2)\in\mathbb{Z}\times\mathbb{Z}:0\leq l_2<l_1\leq d\}
\end{equation*}
and the ``universal'' step-two nilpotent Lie groups $\G_0^\#=\G_0^\#(d)$
\begin{equation*}
\G_0^\#=\{(x_{l_1l_2})_{(l_1,l_2)\in Y_d}:x_{l_1l_2}\in\mathbb{R}\},
\end{equation*}
with the group multiplication law
\begin{equation*}
[x\cdot y]_{l_1l_2}=
\begin{cases}
x_{l_10}+y_{l_10}&\text{ if }l_1\in\{1,\ldots,d\}\text{ and }l_2=0,\\
x_{l_1l_2}+y_{l_1l_2}+x_{l_10}y_{l_20}&\text{ if }l_1\in\{1,\ldots,d\}\text{ and }l_2\in\{1,\ldots,l_1-1\}.
\end{cases}
\end{equation*}

Let $\G_0=\G_0(d)$ denote the discrete subgroup $\G_0=\G_0^\#\cap\Z^{|Y_d|}$. Let $A_0:\mathbb{R}\to\G_0^\#$ denote the polynomial map
\begin{equation}\label{tra3}
[A_0(x)]_{l_1l_2}=
\begin{cases}
x^{l_1}&\text{ if }l_2=0,\\
0&\text{ if }l_2\neq 0,
\end{cases}
\end{equation}
and notice that $A_0(\Z)\subseteq\G_0$.

\begin{lemma}\label{morphism}
Assuming $\G^\#$ and $A$ are defined as before, there is $d$ sufficiently large and a group morphism $T:\G_0\to\G^\#$ such that
\begin{equation}\label{tra9}
A(n)=T(A_0(n))\qquad\text{ for any }n\in\mathbb{Z}.
\end{equation}
\end{lemma}

\begin{proof}[Proof of Lemma \ref{morphism}] Set
\begin{equation*}
d=2d_3
\end{equation*}
and let $g_1,\ldots,g_d$ denote the generators of the group $\mathbb{G}_0$,
\begin{equation*}
[g_m]_{l_1l_2}=\begin{cases}1&\text{ if }l_1=m\text{ and }l_2=0,\\
0&\text{ otherwise}.
\end{cases}
\end{equation*}
We notice that any group morphism $T:\G_0\to\G^\#$ is uniquely determined by the values $ T(g_1),\ldots,T(g_d)$. Indeed, any element
\begin{equation*}
x=(x_{l_1l_2})_{(l_1,l_2)\in Y_d}\in\G_0,\qquad x_{l_1l_2}\in\mathbb{Z},
\end{equation*}
can be written in the form
\begin{equation*}
x=g_1^{x_{10}}\cdot\ldots\cdot g_d^{x_{d0}}\cdot \prod_{1\leq l_2<l_1\leq d}(g_{l_1}g_{l_2}g_{l_1}^{-1}g_{l_2}^{-1})^{x_{l_1l_2}}.
\end{equation*}
Therefore, if $T(g_l)=h_l\in\G^\#$ then $T$ is uniquely defined by
\begin{equation*}
T(x)=h_1^{x_{10}}\cdot\ldots\cdot h_d^{x_{d0}}\cdot \prod_{1\leq l_2<l_1\leq d}(h_{l_1}h_{l_2}h_{l_1}^{-1}h_{l_2}^{-1})^{x_{l_1l_2}},\qquad \text{ if }x=(x_{l_1l_2})_{(l_1,l_2)\in Y_d}\in\G_0.
\end{equation*}
It is easy to verify that this defines indeed a group morphism, using the fact that the elements $h_{l_1}h_{l_2}h_{l_1}^{-1}h_{l_2}^{-1}$ are in the center of the group $\G^\#$.

Assume that
\begin{equation}\label{tra4}
A(n)=\Big(\sum_{i=1}^{d_3}\al_{i}n^i,\sum_{i=1}^{d_3}\beta_{i}n^i\Big),\qquad \alpha_1,\ldots,\alpha_{d_3}\in\mathbb{R}^{d_1},\, \beta_1,\ldots,\beta_{d_3}\in\mathbb{R}^{d_2}.
\end{equation}
We define
\begin{equation*}
T(g_l)=\begin{cases}
(\alpha_l,\gamma_l)&\text{ if }l\in\{1,\ldots,d_3\},\\
(0,\gamma_l)&\text{ if }l\in\{d_3+1,\ldots,d\},
\end{cases}
\end{equation*}
for some vectors $\gamma_1,\ldots,\gamma_d\in \mathbb{R}^{d_2}$ to be fixed, and extend $T$ as a group morphism from $\G_0\to\G^\#$. Since
\begin{equation*}
A_0(n)=g_1^n\cdot\ldots\cdot g_d^{n^d},
\end{equation*}
it follows that
\begin{equation*}
T(A_0(n))=(\sum_{i=1}^{d_3}\alpha_in^i,\sum_{i=1}^d\gamma_in^i+\sum_{i=1}^{2d_3}\rho_in^i),
\end{equation*}
for some coefficients $\rho_1,\ldots\rho_{2d_3}$ that depend only on $(\alpha_i)_{i\in\{1,\ldots,d_3\}}$ and the bilinear form $R$. The desired identity $T(A_0(n))=A(n)$ can be arranged by choosing the vectors $\gamma_1,\ldots,\gamma_d$ appropriately.
\end{proof}

Assume now that we could prove the following  particular case of Theorem \ref{Main1}:

\begin{theorem}\label{Main2}
For any $d\geq 1$, $R\geq 1$, and $F:\mathbb{G}_0\to\mathbb{C}$ let
\begin{equation*}
(H^R_0F)(g_0)=\sum_{|n|\leq R}K(n)F(A_0(n)^{-1}\cdot g_0),
\end{equation*}
where $A_0:\mathbb{Z}\to\G_0$ is as in \eqref{tra3} and $K$ is as in \eqref{CalZyg}. Then
\begin{equation*}
\|H^R_0F\|_{L^2(\G_0)}\lesssim_d \|F\|_{L^2(\G_0)}\qquad\text{ uniformly in }R.
\end{equation*}
\end{theorem}

It is not hard to see that Theorem \ref{Main2} would imply Theorem \ref{Main1}. This follows from the standard transference principle, see \cite[Proposition 5.1]{RiSt2}. Indeed, given a polynomial map $A:\mathbb{R}\to \G^\#$ with $A(0)=0$, we fix a group morphism $T:\G_0\to\G^\#$ such as $A(n)=T(A_0(n))$, $n\in\Z$. Then we define the isometric representation $\pi$ of $\G_0$ on $L^2(\G^\#)$,
\begin{equation}\label{represe}
\pi(g_0)(f)(g)=f(T(g_0^{-1})\cdot g),\qquad g_0\in\G_0,\,f\in L^2(\G^\#),\,g\in\G^\#.
\end{equation}
For $R\geq 1$ we define
\begin{equation*}
\begin{split}
&K^R:\mathbb{Z}\to\mathbb{C},\qquad K^R(n)=K(n)\mathbf{1}_{[-R,R]\cap\mathbb{Z}}(n),\\
&(H^Rf)(g)=\sum_{n\in\mathbb{Z}}K^R(n)f(A(n)^{-1}\cdot g),\qquad f\in C_0(\G^\#).
\end{split}
\end{equation*}
Then, for any bounded open set $U\subseteq\G_0$, $R\geq 1$, and $f\in C_0(\G^\#)$
\begin{equation*}
\|H^Rf\|_{L^2(\G^\#)}^2=\frac{1}{|U|}\int_U\int_{\G^\#}|\pi(g_0^{-1})(H^Rf)(g)|^2\,dgdg_0.
\end{equation*}
The definitions show that
\begin{equation*}
\pi(g_0^{-1})(H^Rf)(g)=(H^Rf)(T(g_0)\cdot g)=\sum_{n\in\Z}K^R(n)f(T(A_0(n)^{-1}\cdot g_0)\cdot g)=H_0^R(F_g)(g_0),
\end{equation*}
where, by definition,
\begin{equation*}
F_g(h_0)=f(T(h_0)\cdot g).
\end{equation*}
Notice that, for $g_0\in U$,
\begin{equation*}
H_0^R(F)(g_0)=H_0^R(F\cdot \mathbf{1}_{U'_R})(g_0),\qquad U'_R=\{(\um,\vm)\cdot h:h\in U,\,|\um|^2+|\vm|<C_dR^{2d}\}.
\end{equation*}
Therefore, using these identities and Theorem \ref{Main2},
\begin{equation*}
\begin{split}
\|H^Rf\|_{L^2(\G^\#)}^2&=\frac{1}{|U|}\int_U\int_{\G^\#}|H_0^R(F_g\cdot\mathbf{1}_{U'_R})(g_0)|^2 \,dgdg_0\\
&\lesssim_d\frac{1}{|U|}\int_{\G_0}\int_{\G^\#}|(F_g\cdot\mathbf{1}_{U'_R})(h_0)|^2\,dgdh_0\\
&\lesssim_d\frac{1}{|U|}\int_{\G_0}\int_{\G^\#}|f(T(h_0)\cdot g)|^2\cdot\mathbf{1}_{U'_R}(h_0)\,dgdh_0\\
&\lesssim_d\frac{|U'_R|}{|U|}\|f\|^2_{L^2(\G^\#)}.
\end{split}
\end{equation*}
For $R$ fixed we can fix $U$ large enough such that $|U'_R|/|U|\leq 2$. Thus $\|H^Rf\|_{L^2(\G)}\lesssim_d\|f\|_{L^2(\G)}$ uniformly in $R$, as desired.

The rest of the paper is concerned with the proof of Theorem \ref{Main2}. We will assume from now on that $d$ is fixed, and all the implied constants are allowed to depend on $d$.

\section{The main kernels: identities and estimates}\label{oscl2}

We fix $\eta_0:\mathbb{R}\to[0,1]$ a smooth even function supported in the interval $[-2,2]$ and equal to $1$ in the interval $[-1,1]$. We define
\begin{equation*}
\eta_j(t)=\eta_0(2^{-j}t)-\eta_0(2^{-j+1}t),\qquad t\in\mathbb{R},\,j=1,2,\ldots,\qquad1=\sum_{j=0}^\infty \eta_j.
\end{equation*}
For $\lambda\geq 1$ let $\widetilde{\eta}_{\leq \lambda}:\mathbb{R}^{|Y_d|}\to[0,1]$,
\begin{equation*}
\widetilde{\eta}_{\leq \lambda}(x)=\prod_{(l_1,l_2)\in Y_d}\eta_0(x_{l_1l_2}/2^{\lambda(l_1+l_2)}).
\end{equation*}
For $x=(x_{l_1l_2})_{(l_1,l_2)\in Y_d}\in\mathbb{R}^{|Y_d|}$ and $\Lambda\in(0,\infty)$ let
\begin{equation*}
\Lambda\circ x=(\Lambda^{l_1+l_2}x_{l_1l_2})_{(l_1,l_2)\in Y_d}\in\mathbb{R}^{|Y_d|},\qquad |x|=\sum_{(l_1,l_2)\in Y_d}|x_{l_1l_2}|.
\end{equation*}
Let
\begin{equation*}
\mathcal{D}^\#_\Lambda=\{x\in \mathbb{R}^{|Y_d|}:|(1/\Lambda)\circ x|<1\},\qquad \mathcal{D}_\Lambda=\mathcal{D}^\#_\Lambda\cap\Z^{|Y_d|}.
\end{equation*}

For $j=1,2,\ldots$ let
\begin{equation}\label{mc1}
\begin{split}
&K_j(t)=K(t)\eta_j(t)+c_j2^{-j}\eta_j(t)-c_{j+1}2^{-j-1}\eta_{j+1}(t),\\
&\text{ where }\,\,c_j=2\Big(\int_{\mathbb{R}}K(t)\big[\sum_{k=0}^{j-1}\eta_k(t)\big]\,dt\Big)\Big(\int_{\mathbb{R}}\eta_0(t)\,dt\Big)^{-1}.
\end{split}
\end{equation}
Using this definition and the assumption \eqref{CalZyg}, it follows that, for $j=1,2,\ldots$
\begin{equation}\label{mc2}
\begin{split}
&2^j|K_j(t)|+2^{2j}|K'_j(t)|\lesssim \mathbf{1}_{[-2^{j+3},2^{j+3}]}(t),\qquad\int_{\mathbb{R}}K_j(t)\,dt=0,\qquad \sup_{j=1,2,\ldots}|c_j|\lesssim 1,\\
&\sum_{j'=1}^{j}K_{j'}(t)=K(t)\eta_0(2^{-j}t)-K(t)\eta_0(t)+c_12^{-1}\eta_1(t)-c_{j+1}2^{-j-1}\eta_{j+1}(t).
\end{split}
\end{equation}
For $f\in L^2(\G_0)$ let
\begin{equation*}
(H_jf)(g)=\sum_{n\in\Z}K_j(n)f(A_0(n)^{-1}\cdot g).
\end{equation*}

In this section we use the notation and the estimates in section \ref{oscil}, in particular Proposition \ref{minarcs} and Lemma \ref{intest}. Any vector in $\mathbb{Q}^m$ has a unique representation in the form $a/q$, with $q\in\{1,2,\ldots\}$, $a\in\Z^m$, and $(a,q)=1$. For $R\in[1,\infty]$ let $\mathcal{S}_R$ denote the set of irreducible fractions in $(\mathbb{Q}\cap(0,1])^{|Y_d|}$ with denominators $\leq R$, i.e.
\begin{equation*}
\mathcal{S}_R=\{a/q=(a_{l_1l_2}/q)_{(l_1,l_2)\in Y_d}:\,1\leq q\leq R,\,a_{l_1l_2}\in Z_q,\,(a,q)=1\}.
\end{equation*}
We fix once and for all three parameters $\eps,r,\kappa$, $0<\kappa\ll1/r\ll\eps\ll 1$, $r\in 2^{\Z_+}$, depending only on $d$ and satisfying
\begin{equation}\label{mc9.5}
 \eps=\overline{C}^{-1}(10d)^{-10},\qquad -2\overline{C}+r\eps/(2\overline{C})\geq (10d)^{10},\qquad \kappa r^2=1,
\end{equation}
where $\overline{C}$ is the constant in Proposition \ref{minarcs} and Lemma \ref{intest}.

For $a/q\in \mathcal{S}_{\infty}$ let
\begin{equation}\label{mc11}
S(a/q)=q^{-2r}\sum_{v,w\in Z_q^r}e^{-2\pi iD(v,w)\cdot a/q},\qquad \widetilde{S}(a/q)=q^{-2r}\sum_{v,w\in Z_q^r}e^{-2\pi i\widetilde{D}(v,w)\cdot a/q}.
\end{equation}
where $D,\widetilde{D}$ are defined in \eqref{pro0.4} and \eqref{pro0.5}.

\begin{lemma}\label{Saq}
For any $a/q\in \mathcal{S}_{\infty}$
\begin{equation}\label{mc7}
|S(a/q)|+|\widetilde{S}(a/q)|\lesssim q^{-(10d)^{10}}.
\end{equation}
\end{lemma}

\begin{proof}[Proof of Lemma \ref{Saq}] For $(l_1,l_2)\in Y_d$ we write
\begin{equation*}
\frac{a_{l_1l_2}}{q}=\frac{a'_{l_1l_2}}{q_{l_1l_2}},\qquad (a'_{l_1l_2},q_{l_1l_2})=1,\qquad 1\leq q_{l_1l_2}\leq q.
\end{equation*}
The bound follows from Proposition \ref{minarcs} with $P=q$ if
\begin{equation*}
\text{ there is }(l_1,l_2)\in Y_d\text{ such that }q^\eps\leq q_{l_1l_2}\leq q^{l_1+l_2-\eps}.
\end{equation*}
Otherwise, since
\begin{equation*}
\sup_{(l_1,l_2)\in Y_d}q_{l_1l_2}\leq q\leq\prod_{(l_1,l_2)\in Y_d}q_{l_1l_2},
\end{equation*}
we necessarily have
\begin{equation}\label{mc100}
q_{l_1l_2}\leq q^\eps\text{ if }l_1+l_2\geq 2\text{ and }q_{10}\geq q^{1-\eps}.
\end{equation}
In this case we may assume $q\geq 2$, and let $Q$ denote the smallest common multiple of $q_{l_1l_2}$, $l_1+l_2\geq 2$, $q/Q\in \{2,3,\ldots\}$. Then we estimate, using the formula \eqref{pro0.4},
\begin{equation*}
\begin{split}
|S(a/q)|&\leq q^{-r}\sup_{v\in Z_q^r}\Big|\sum_{w\in Z_q^r}e^{-2\pi iD(v,w)\cdot a/q}\Big|\leq q^{-r}\sup_{v\in Z_q^r}\sum_{y\in Z_Q^r}\Big|\sum_{x\in Z_{q/Q}^r}e^{-2\pi iD(v,Qx+y)\cdot a/q}\Big|\\
&\leq q^{-r}\sup_{v\in Z_q^r}\sum_{y\in Z_Q^r}\Big|\sum_{x\in Z_{q/Q}^r}e^{-2\pi iD(v,Qx+y)_{10}\cdot a_{10}/q_{10}}\Big|=0,
\end{split}
\end{equation*}
which suffices. The bound on $|\widetilde{S}(a/q)|$ is similar.
\end{proof}

The main goal in this section is to describe the operators
\begin{equation*}
H_{j_1}H_{j_2}^\ast\ldots H_{j_{2r-1}}H^\ast_{j_{2r}}\qquad\text{ and }\qquad H^\ast_{j_1}H_{j_2}\ldots H_{j_{2r-1}}^\ast H_{j_{2r}},
\end{equation*}
for suitable values of $j_1,\ldots,j_{2r}$. More precisely, we prove the following:

\begin{proposition}\label{majarcs}
Assume $C(d)$ is a sufficiently large constant, $J\in[C(d),\infty)$, and $j_1,k_1\ldots,j_{r},k_r\in[J(1-\kappa),J]\cap\Z$. Then
\begin{equation*}
\begin{split}
&(H_{j_1}^\ast H_{k_1}\ldots H_{j_r}^\ast H_{k_r}F)(g)=\sum_{h\in \G_0}[K_{j_1,k_1,\ldots,j_r,k_r}(h)+E_{j_1,k_1,\ldots,j_r,k_r}(h)]F(h^{-1}\cdot g),\\
&(H_{j_1}H_{k_1}^\ast\ldots H_{j_r}H_{k_r}^\ast F)(g)=\sum_{h\in \G_0}[\widetilde{K}_{j_1,k_1,\ldots,j_r,k_r}(h)+\widetilde{E}_{j_1,k_1,\ldots,j_r,k_r}(h)]F(h^{-1}\cdot g),
\end{split}
\end{equation*}
for any $F\in L^2(\G_0)$ and $g\in\G_0$, where
\begin{equation}\label{mc5}
\|E_{j_1,k_1,\ldots,j_r,k_r}\|_{L^1(\G_0)}+\|\widetilde{E}_{j_1,k_1,\ldots,j_r,k_r}\|_{L^1(\G_0)}\lesssim 2^{-J/4}.
\end{equation}
Moreover
\begin{equation}\label{mc6.1}
\begin{split}
&K_{j_1,k_1,\ldots,j_r,k_r}(h)=\widetilde{\eta}_{\leq J+\eps J}(h)\sum_{a/q\in\mathcal{S}_{2^{3d^2\eps J}}}e^{2\pi ih\cdot a/q}S(a/q)\\
&\int_{\mathbb{R}^{|Y_d|}}\int_{\mathbb{R}^r\times\mathbb{R}^r}\prod_{(l_1,l_2)\in Y_d}\eta_0(2^{J(l_1+l_2-2\eps)}\beta_{l_1l_2})G_{j_1,k_1,\ldots,j_r,k_r}(x,y)e^{2\pi i (h-D(x,y))\cdot\beta}\,dxdyd\beta,
\end{split}
\end{equation}
\begin{equation}\label{mc6.2}
\begin{split}
&\widetilde{K}_{j_1,k_1,\ldots,j_r,k_r}(h)=\widetilde{\eta}_{\leq J+\eps J}(h)\sum_{a/q\in\mathcal{S}_{2^{3d^2\eps J}}}e^{2\pi ih\cdot a/q}\widetilde{S}(a/q)\\
&\int_{\mathbb{R}^{|Y_d|}}\int_{\mathbb{R}^r\times\mathbb{R}^r}\prod_{(l_1,l_2)\in Y_d}\eta_0(2^{J(l_1+l_2-2\eps)}\beta_{l_1l_2})G_{j_1,k_1,\ldots,j_r,k_r}(x,y)e^{2\pi i (h-\widetilde{D}(x,y))\cdot\beta}\,dxdyd\beta.
\end{split}
\end{equation}
\end{proposition}

The functions $G_{j_1,k_1,\ldots,j_r,k_r}$ are defined by
\begin{equation*}
G_{j_1,k_1,\ldots,j_r,k_r}(x,y)=K_{j_1}(x_1)K_{k_1}(y_1)\ldots K_{j_r}(x_r)K_{k_r}(y_r),\quad x,y\in\mathbb{R}^r.
\end{equation*}
The functions $D,\widetilde{D}:\mathbb{R}^r\times\mathbb{R}^r\to\mathbb{R}^{|Y_d|}$ are defined in \eqref{pro0.4} and \eqref{pro0.5}.

\begin{proof}[Proof of Proposition \ref{majarcs}] We only prove the claims for the operators $H_{j_1}^\ast H_{k_1}\ldots H_{j_r}^\ast H_{k_r}$ and the kernels $K_{j_1,k_1,\ldots,j_r,k_r},E_{j_1,k_1,\ldots,j_r,k_r}$; the claims for the  operators $H_{j_1} H_{k_1}^\ast\ldots H_{j_r}H_{k_r}^\ast$ and the kernels $\widetilde{K}_{j_1,k_1,\ldots,j_r,k_r},\widetilde{E}_{j_1,k_1,\ldots,j_r,k_r}$ follow by essentially identical arguments. Recall that $\eps,r,\kappa$ are fixed, depending only on $d$, so all the implicit constants are allowed to depend on $\eps,r,\kappa$.

By definition,
\begin{equation}\label{mc7.5}
\begin{split}
(H_{j_1}^\ast H_{k_1}\ldots H_{j_r}^\ast H_{k_r}F)(g)&=\sum_{n_1,m_1,\ldots,n_r,m_r\in\Z}K_{j_1}(n_1)K_{k_1}(m_1)\ldots K_{j_r}(n_r)K_{k_r}(m_r)\\
&F(A_0(m_r)^{-1}\cdot A_0(n_r)\cdot\ldots \cdot A_0(m_1)^{-1}\cdot A_0(n_1)\cdot g).
\end{split}
\end{equation}
Recalling the definition \eqref{pro0.3} and letting
\begin{equation}\label{mc8}
\begin{split}
L_{j_1,k_1,\ldots,j_r,k_r}(h)=\widetilde{\eta}_{\leq J+\eps J}(h)\int_{[0,1]^{|Y_d|}}&\sum_{n,m\in\Z^r}G_{j_1,k_1,\ldots,j_r,k_r}(n,m)\\
&e^{2\pi i\sum_{(l_1,l_2)\in Y_d}(h_{l_1l_2}-D(n,m)_{l_1l_2})\theta_{l_1l_2}}\,d\theta,
\end{split}
\end{equation}
this becomes
\begin{equation*}
(H_{j_1}^\ast H_{k_1}\ldots H_{j_r}^\ast H_{k_r}F)(g)=\sum_{h\in \G_0}L_{j_1,k_1,\ldots,j_r,k_r}(h)F(h^{-1}\cdot g).
\end{equation*}
It remains to prove that we can decompose $L_{j_1,k_1,\ldots,j_r,k_r}=K_{j_1,k_1,\ldots,j_r,k_r}+E_{j_1,k_1,\ldots,j_r,k_r}$ satisfying the claims in the proposition.

We decompose the integral over $\theta$ in \eqref{mc8} into the contribution of major and minor arcs. Let
\begin{equation}\label{mc9}
\begin{split}
L^1_{j_1,k_1,\ldots,j_r,k_r}&(h)=\widetilde{\eta}_{\leq J+\eps J}(h)\sum_{a/q\in\mathcal{S}_{2^{3d^2\eps J}}}\int_{\mathbb{R}^{|Y_d|}}\sum_{n,m\in\Z^r}G_{j_1,k_1,\ldots,j_r,k_r}(n,m)\\
&e^{2\pi i\sum_{(l_1,l_2)\in Y_d}(h_{l_1l_2}-D(n,m)_{l_1l_2})(a_{l_1l_2}/q+\beta_{l_1l_2})}\prod_{(l_1,l_2)\in Y_d}\eta_0(2^{J(l_1+l_2-2\eps)}\beta_{l_1l_2})\,d\beta.
\end{split}
\end{equation}
In view of the choice of $\eps,r,\kappa$ and the restriction $j_1,k_1,\ldots,j_r,k_r\in[(1-\kappa)J,J]$, it follows from Proposition \ref{minarcs} and \eqref{mc2} that
\begin{equation*}
|L_{j_1,k_1,\ldots,j_r,k_r}(h)-L^1_{j_1,k_1,\ldots,j_r,k_r}(h)|\lesssim 2^{-10d^2J}\widetilde{\eta}_{\leq J+\eps J}(h),\qquad h\in\G_0,
\end{equation*}
which is consistent with the error estimate \eqref{mc5}.

We consider now the sum over $m,n$ in \eqref{mc9}, and rewrite, for $q$ fixed,
\begin{equation*}
\begin{split}
&\sum_{n,m\in\Z^r}G_{j_1,k_1,\ldots,j_r,k_r}(n,m)e^{-2\pi i D(n,m)\cdot(a/q+\beta)}\\
&=\sum_{n,m\in\Z^r}\sum_{v,w\in Z_q^r}G_{j_1,k_1,\ldots,j_r,k_r}(qn+v,qm+w)e^{-2\pi i D(qn+v,qm+w)\cdot\beta}e^{-2\pi iD(v,w)\cdot a/q}\\
&=E'(a/q,\beta)+\sum_{n,m\in\Z^r}\sum_{v,w\in Z_q^r}G_{j_1,k_1,\ldots,j_r,k_r}(qn,qm)e^{-2\pi i D(qn,qm)\cdot\beta}e^{-2\pi iD(v,w)\cdot a/q}.
\end{split}
\end{equation*}
For $q\leq 2^{3d^2\eps J}$ and $\beta=(\beta_{l_1l_2})_{(l_1,l_2)\in Y_d}$, $|\beta_{l_1l_2}|\leq 22^{-J(l_1+l_2-2\eps)}$, we estimate, using \eqref{mc2}, \eqref{pro0.4}, and the assumption $j_1,k_1,\ldots,j_r,k_r\in[(1-\kappa)J,J]$,
\begin{equation*}
|E'(a/q,\beta)|\lesssim 2^{-3J/4}.
\end{equation*}
Therefore, if we define
\begin{equation}\label{mc10}
\begin{split}
L^2_{j_1,k_1,\ldots,j_r,k_r}(h)=&\widetilde{\eta}_{\leq J+\eps J}(h)\sum_{a/q\in\mathcal{S}_{2^{3d^2\eps J}}}\int_{\mathbb{R}^{|Y_d|}}e^{2\pi ih\cdot(a/q+\beta)}\prod_{(l_1,l_2)\in Y_d}\eta_0(2^{J(l_1+l_2-2\eps)}\beta_{l_1l_2})\\
&\sum_{n,m\in\Z^r}\sum_{v,w\in Z_q^r}G_{j_1,k_1,\ldots,j_r,k_r}(qn,qm)e^{-2\pi i D(qn,qm)\cdot\beta}e^{-2\pi iD(v,w)\cdot a/q}\,d\beta,
\end{split}
\end{equation}
it follows that
\begin{equation*}
|L^1_{j_1,k_1,\ldots,j_r,k_r}(h)-L^2_{j_1,k_1,\ldots,j_r,k_r}(h)|\lesssim 2^{-J/2}\prod_{(l_1,l_2)\in Y_d}2^{-J(l_1+l_2)}\widetilde{\eta}_{\leq J+\eps J}(h).
\end{equation*}
This is consistent with the error estimate in \eqref{mc5}.

Finally, it remains to decompose the kernel $L^2_{j_1,k_1,\ldots,j_r,k_r}$. For this we rewrite first
\begin{equation*}
\begin{split}
L^2_{j_1,k_1,\ldots,j_r,k_r}(h)=&\widetilde{\eta}_{\leq J+\eps J}(h)\sum_{a/q\in\mathcal{S}_{2^{3d^2\eps J}}}e^{2\pi ih\cdot a/q}S(a/q)\int_{\mathbb{R}^{|Y_d|}}e^{2\pi ih\cdot\beta}\prod_{(l_1,l_2)\in Y_d}\eta_0(2^{J(l_1+l_2-2\eps)}\beta_{l_1l_2})\\
&q^{2r}\sum_{n,m\in\Z^r}G_{j_1,k_1,\ldots,j_r,k_r}(qn,qm)e^{-2\pi i D(qn,qm)\cdot\beta}\,d\beta,
\end{split}
\end{equation*}
where $S(a/q)$ is defined in \eqref{mc11}. Using the formula \eqref{pro0.4}, we estimate for any $q\leq 2^{3d^2\eps J}$ and $\beta=(\beta_{l_1l_2})_{(l_1,l_2)\in Y_d}$, $|\beta_{l_1l_2}|\leq 22^{-J(l_1+l_2-2\eps)}$,
\begin{equation*}
\begin{split}
\Big|\sum_{n,m\in\Z^r}G_{j_1,k_1,\ldots,j_r,k_r}(qn,qm)e^{-2\pi i D(qn,qm)\cdot\beta}-\int_{\mathbb{R}^r\times\mathbb{R}^r}G_{j_1,k_1,\ldots,j_r,k_r}(qx,qy)e^{-2\pi i D(qx,qy)\cdot\beta}\,dxdy\Big|&\\
\lesssim q^{-2r}2^{-3J/4}.&
\end{split}
\end{equation*}
Thus, with $K_{j_1,k_1,\ldots,j_r,k_r}$ defined as in \eqref{mc6.1}, we have the pointwise bound
\begin{equation*}
|L^2_{j_1,k_1,\ldots,j_r,k_r}(h)-K_{j_1,k_1,\ldots,j_r,k_r}(h)|\lesssim 2^{-J/2}\prod_{(l_1,l_2)\in Y_d}2^{-J(l_1+l_2)}\widetilde{\eta}_{\leq J+\eps J}(h),
\end{equation*}
which is consistent with the error estimate in \eqref{mc5}. This completes the proof of the proposition.
\end{proof}

Assume $J,j_1,k_1,\ldots,j_r,k_r$ are as in Proposition \ref{majarcs} and define
\begin{equation}\label{mc60}
\begin{split}
&P_{j_1,k_1,\ldots,j_r,k_r}(v)=\int_{\mathbb{R}^r\times\mathbb{R}^r} G_{j_1,k_1,\ldots,j_r,k_r}(x,y) e^{-2\pi iD(x,y)\cdot(2^{-J}\circ v)}\,dxdy,\\
&\widetilde{P}_{j_1,k_1,\ldots,j_r,k_r}(v)=\int_{\mathbb{R}^r\times\mathbb{R}^r} G_{j_1,k_1,\ldots,j_r,k_r}(x,y) e^{-2\pi i\widetilde{D}(x,y)\cdot (2^{-J}\circ v)}\,dxdy.
\end{split}
\end{equation}
Notice that the formulas \eqref{mc6.1} and \eqref{mc6.2} become, after changes of variables
\begin{equation}\label{mc6.3}
\begin{split}
K_{j_1,k_1,\ldots,j_r,k_r}(h)&=\widetilde{\eta}_{\leq J+\eps J}(h)\prod_{(l_1,l_2)\in Y_d}2^{-J(l_1+l_2)}\sum_{a/q\in\mathcal{S}_{2^{3d^2\eps J}}}e^{2\pi ih\cdot a/q}S(a/q)\\
&\int_{\mathbb{R}^{|Y_d|}}\prod_{(l_1,l_2)\in Y_d}\eta_0(2^{-2\eps J}v_{l_1l_2})P_{j_1,k_1,\ldots,j_r,k_r}(v)e^{2\pi i (2^{-J}\circ h)\cdot v}\,dv,
\end{split}
\end{equation}
and
\begin{equation}\label{mc6.4}
\begin{split}
\widetilde{K}_{j_1,k_1,\ldots,j_r,k_r}(h)&=\widetilde{\eta}_{\leq J+\eps J}(h)\prod_{(l_1,l_2)\in Y_d}2^{-J(l_1+l_2)}\sum_{a/q\in\mathcal{S}_{2^{3d^2\eps J}}}e^{2\pi ih\cdot a/q}\widetilde{S}(a/q)\\
&\int_{\mathbb{R}^{|Y_d|}}\prod_{(l_1,l_2)\in Y_d}\eta_0(2^{-2\eps J}v_{l_1l_2})\widetilde{P}_{j_1,k_1,\ldots,j_r,k_r}(v)e^{2\pi i (2^{-J}\circ h)\cdot v}\,dv.
\end{split}
\end{equation}
In view of the cancellation condition in the first line of \eqref{mc2},
\begin{equation}\label{mc61}
P_{j_1,k_1,\ldots,j_r,k_r}(0)=\widetilde{P}_{j_1,k_1,\ldots,j_r,k_r}(0)=0.
\end{equation}
We make the changes of variables $x=2^J\mu$, $y=2^{J}\nu$ to rewrite
\begin{equation*}
\begin{split}
&P_{j_1,k_1,\ldots,j_r,k_r}(v)=\int_{\mathbb{R}^r\times\mathbb{R}^r} 2^{2rJ}G_{j_1,k_1,\ldots,j_r,k_r}(2^J\mu,2^J\nu) e^{-2\pi iD(\mu,\nu)\cdot v}\,d\mu d\nu,\\
&\widetilde{P}_{j_1,k_1,\ldots,j_r,k_r}(v)=\int_{\mathbb{R}^r\times\mathbb{R}^r} 2^{2rJ}G_{j_1,k_1,\ldots,j_r,k_r}(2^J\mu,2^J\nu) e^{-2\pi i\widetilde{D}(\mu,\nu)\cdot v}\,d\mu d\nu.
\end{split}
\end{equation*}
Using Lemma \ref{intest}, for $m=0,1,\ldots$
\begin{equation}\label{mc62}
\begin{split}
&|\nabla^m_v P_{j_1,k_1,\ldots,j_r,k_r}(v)|+|\nabla^m_v \widetilde{P}_{j_1,k_1,\ldots,j_r,k_r}(v)|\lesssim_m 2^{8r(J-\min(j_1,\ldots,j_r,k_1,\ldots,k_r))}(1+|v|)^{-(10d)^{10}}.
\end{split}
\end{equation}

\section{Proof of Theorem \ref{Main2}}\label{proofthm}

In this section we complete the proof of Theorem \ref{Main2}. The main ingredients are Lemma \ref{prop1} and the estimates and the identities proved in section \ref{oscl2}. We use the notation introduced  in section \ref{oscl2}. In view of the identity in the second line of \eqref{mc2}, it suffices to prove that for any integer $J\geq 1$
\begin{equation*}
\big\|\sum_{j=1}^JH_j\big\|_{L^2(\G_0)\to L^2(\G_0)}\lesssim 1.
\end{equation*}
By further dividing into finitely many sums, it suffices to prove the following:

\begin{proposition}\label{separated}
Assume $J_1,\ldots,J_K\in[1,\infty)$ satisfy the separation condition
\begin{equation}\label{tj1}
J_{m+1}\geq 2J_{m},\qquad m=1,\ldots,K-1.
\end{equation}
For $m=1,\ldots,K$ let
\begin{equation*}
S_m=\sum_{j\in[J_m(1-\kappa),J_m]\cap\Z}H_j.
\end{equation*}
Then
\begin{equation*}
\big\|S_1+\ldots+S_K\big\|_{L^2(\G_0)\to L^2(\G_0)}\lesssim 1.
\end{equation*}
\end{proposition}

The rest of the section is concerned with the proof of Proposition \ref{separated}. We would like to apply Lemma \ref{prop1}, in the simplified form given in Remark \ref{prop2}. We will verify the conditions \eqref{na99} in several steps.

\begin{lemma}\label{separated2}
We have
\begin{equation*}
\sup_{J\geq 1}\sup_{A\subseteq[J/2,J]\cap\Z}\big\|\sum_{j\in A}H_j\big\|_{L^2(\G_0)\to L^2(\G_0)}\lesssim 1.
\end{equation*}
\end{lemma}

\begin{proof}[Proof of Lemma \ref{separated2}] In view of the Cotlar--Stein lemma, it suffices to prove that, for some $\delta'>0$,
\begin{equation*}
\|H_k H_j^\ast\|_{L^2\to L^2}+\|H_k^\ast H_j\|_{L^2\to L^2}\lesssim 2^{-\delta'(j-k)}\qquad \text{ for any }k\leq j\in[J/2,J]\cap\Z.
\end{equation*}
Since $\|H_j\|_{L^2\to L^2}\lesssim 1$ for any $j$, it follows that
\begin{equation*}
\begin{split}
&\|H_k H_j^\ast\|_{L^2\to L^2}\lesssim \|H_k H_j^\ast H_j\|^{1/2}_{L^2\to L^2}\lesssim \|H_k (H_j^\ast H_j)^2\|^{1/4}_{L^2\to L^2}\lesssim \ldots\lesssim \|H_k (H_j^\ast H_j)^r\|^{1/(2r)}_{L^2\to L^2},\\
&\|H_k^\ast H_j\|_{L^2\to L^2}\lesssim \|H_k^\ast H_j H_j^\ast\|^{1/2}_{L^2\to L^2}\lesssim \|H_k^\ast (H_j H_j^\ast)^2\|^{1/4}_{L^2\to L^2}\lesssim \ldots\lesssim \|H_k^\ast (H_j H_j^\ast)^r\|^{1/(2r)}_{L^2\to L^2}.
\end{split}
\end{equation*}
Therefore it suffices to prove that there is $\delta=\delta(d)>0$ such that
\begin{equation}\label{tj2}
\|H_k (H_j^\ast H_j)^r\|_{L^2\to L^2}+\|H_k^\ast (H_jH_j^\ast)^r\|_{L^2\to L^2}\lesssim 2^{-\delta(j-k)}
\end{equation}
for any $k,j\in[C(d),\infty)\cap\Z$, $k\in[j/2,j]$.

We will prove only the bound on the first term in the left-hand side of \eqref{tj2}; the bound on the second term is very similar. We use Proposition \ref{majarcs} with $J=j_1=k_1=\ldots=j_r=k_r=j$. With the notation in Proposition \ref{majarcs}
\begin{equation*}
[H_k (H_j^\ast H_j)^r](F)(g)=\sum_{h\in\G_0}F(h^{-1}\cdot g)\sum_{n\in\Z}K_k(n)(K_{j,j,\ldots,j,j}+E_{j,j,\ldots,j,j})(A_0(n)^{-1}\cdot h),
\end{equation*}
for any $F\in L^2(\G_0)$ and $g\in \G_0$. In view of \eqref{mc5}, it suffices to prove that
\begin{equation*}
\Big\|\sum_{n\in\Z}K_k(n)K_{j,j,\ldots,j,j}(A_0(n)^{-1}\cdot h)\Big\|_{L^1_h(\G_0)}\lesssim 2^{-\delta(j-k)}.
\end{equation*}
We use now the formula \eqref{mc6.3}. For $x\in\mathbb{R}^{|Y_d|}$ let
\begin{equation}\label{tj4}
\begin{split}
M_j(x)&=\widetilde{\eta}_{\leq j+\eps j}(x)\prod_{(l_1,l_2)\in Y_d}2^{-j(l_1+l_2)}\\
&\int_{\mathbb{R}^{|Y_d|}}\prod_{(l_1,l_2)\in Y_d}\eta_0(2^{-2\eps j}\beta_{l_1l_2})P_{j,j,\ldots,j,j}(\beta)e^{2\pi i (2^{-j}\circ x)\cdot\beta}\,d\beta.
\end{split}
\end{equation}
Recalling the rapid decay of the coefficients $S(a/q)$ (see Lemma \ref{Saq}), it suffices to prove that for any $a/q\in\mathcal{S}_{2^{3d^2\eps j}}$
\begin{equation}\label{tj5}
\Big\|\sum_{n\in\Z}K_k(n)e^{2\pi i(A_0(n)^{-1}\cdot h)\cdot a/q}M_{j}(A_0(n)^{-1}\cdot h)\Big\|_{L^1_h(\G_0)}\lesssim 2^{-\delta(j-k)}q^{(4d)^4}.
\end{equation}

Using \eqref{mc62} and integration by parts
\begin{equation}\label{tj6}
|M_j(x)|+\sum_{(l_1,l_2)\in Y_d}2^{j(l_1+l_2)}|\partial_{x_{l_1,l_2}}M_j(x)|\lesssim(1+|2^{-j}\circ x|)^{-(4d)^4}\prod_{(l_1,l_2)\in Y_d}2^{-j(l_1+l_2)}.
\end{equation}
Therefore, if $|n|\lesssim 2^k$ and $h\in \G_0$
\begin{equation*}
|M_{j}(A_0(n)^{-1}\cdot h)-M_j(h)|\lesssim 2^{k-j}(1+|2^{-j}\circ h|)^{-(4d)^4}\prod_{(l_1,l_2)\in Y_d}2^{-j(l_1+l_2)}.
\end{equation*}
Thus
\begin{equation}\label{tj7}
\Big\|\sum_{n\in\Z}K_k(n)e^{2\pi i(A_0(n)^{-1}\cdot h)\cdot a/q}[M_{j}(A_0(n)^{-1}\cdot h)-M_j(h)]\Big\|_{L^1_h(\G_0)}\lesssim 2^{k-j}.
\end{equation}

On the other hand, using \eqref{mc2} and the assumption $k\geq j/2$, for any $h\in \G_0$ and $a/q\in\mathcal{S}_{2^{3d^2\eps j}}$
\begin{equation}\label{tj8}
\begin{split}
&\Big|\sum_{n\in\Z}K_k(n)e^{2\pi i(A_0(n)^{-1}\cdot h)\cdot a/q}\Big|\leq\sum_{m\in Z_q}\Big|\sum_{n\in\Z}K_k(qn+m)e^{2\pi i(A_0(qn+m)^{-1}\cdot h)\cdot a/q}\Big|\\
&\leq \sum_{m\in Z_q}\Big|\sum_{n\in\Z}K_k(qn+m)\Big|\lesssim 2^{-j/4}.
\end{split}
\end{equation}
Thus, using also \eqref{tj6},
\begin{equation}\label{tj9}
\Big\|\sum_{n\in\Z}K_k(n)e^{2\pi i(A_0(n)^{-1}\cdot h)\cdot a/q}M_j(h)\Big\|_{L^1_h(\G_0)}\lesssim 2^{-j/4},
\end{equation}
and the bound \eqref{tj5} follows from \eqref{tj7} and \eqref{tj9}. This completes the proof.
\end{proof}

\begin{remark}\label{rema}
We observe that it is important to assume that $j/k\lesssim 1$ in the proof of the bound \eqref{tj2}. Otherwise one could only prove a weaker bound, of the form
\begin{equation*}
\|H_k H_j^\ast\|_{L^2\to L^2}+\|H_k^\ast H_j\|_{L^2\to L^2}\lesssim 2^{-\delta'(j-k)}+2^{-\delta' k}\qquad \text{ for any }k\leq j\in\{1,2,\ldots\}.
\end{equation*}
Such a bound does not suffice to apply the Cotlar--Stein lemma to prove the theorem directly. It is precisely to compensate for this failure that we need the additional orthogonality proposition in section \ref{orthog}.
\end{remark}

We consider now long sums of operators $(S_m^\ast S_m)^r$ and $(S_m S_m^\ast)^r$.

\begin{lemma}\label{separated3}
Assume $J_1,\ldots,J_K\in[C(d),\infty)$ satisfy the separation condition
\begin{equation}\label{tj20}
J_{m+1}\geq 2J_{m},\qquad m=1,\ldots,K-1.
\end{equation}
For $m=1,\ldots,K$ let
\begin{equation*}
S_m=\sum_{j\in[J_m(1-\kappa),J_m]\cap\Z}H_j.
\end{equation*}
Then
\begin{equation}\label{tj21}
\big\|(S_1^\ast S_1)^r+\ldots+(S_K^\ast S_K)^r\big\|_{L^2\to L^2}+\big\|(S_1 S_1^\ast)^r+\ldots+(S_K S_K^\ast)^r\big\|_{L^2\to L^2}\lesssim 1.
\end{equation}
\end{lemma}

\begin{proof}[Proof of Lemma \ref{separated3}] We prove only the bound on the first term in the left-hand side of \eqref{tj21}. In view of Proposition \ref{majarcs}, it suffices to prove that
\begin{equation*}
\Big\|\sum_{m=1}^KF\ast\Big[\sum_{j_1,k_1,\ldots,j_r,k_r\in[J_m(1-\kappa),J_m]\cap\Z}K_{j_1,k_1,\ldots,j_r,k_r}\Big]\Big\|_{L^2(\G_0)}\lesssim \|F\|_{L^2(\G_0)}
\end{equation*}
for any $F\in L^2(\G_0)$. For $x\in\mathbb{R}^{|Y_d|}$ and $m=1,\ldots,K$ we define
\begin{equation}\label{tj22}
\begin{split}
&N_m(x)=\widetilde{\eta}_{\leq J_m+\eps J_m}(x)\prod_{(l_1,l_2)\in Y_d}2^{-J_m(l_1+l_2)}\\
&\int_{\mathbb{R}^{|Y_d|}}\prod_{(l_1,l_2)\in Y_d}\eta_0(2^{-2\eps J_m}\beta_{l_1l_2})\sum_{j_1,k_1,\ldots,j_r,k_r\in[J_m(1-\kappa),J_m]\cap\Z}P_{j_1,k_1,\ldots,j_r,k_r}(\beta)e^{2\pi i (2^{-J_m}\circ x)\cdot\beta}\,d\beta.
\end{split}
\end{equation}
We use the formula \eqref{mc6.3} and the rapid decay of the coefficients $S(a/q)$ in Lemma \ref{Saq}. After rearranging the sum, it suffices to prove that for any $a/q\in \mathcal{S}_\infty$
\begin{equation}\label{tj23}
\Big\|\sum_{2^{J_m}\geq q^{(8d)^8}}\sum_{h\in \G_0}F(h^{-1}\cdot g)e^{2\pi i h\cdot a/q}N_m(h)\Big\|_{L^2_g(\G_0)}\lesssim q^{(4d)^4}\|F\|_{L^2(\G_0)}
\end{equation}
for any $F\in L^2(\G_0)$.

Using \eqref{mc62} and integration by parts
\begin{equation}\label{tj22.4}
|N_m(x)|+\sum_{(l_1,l_2)\in Y_d}2^{J_m(l_1+l_2)}|\partial_{x_{l_1l_2}}N_m(x)|\lesssim_C 2^{4\eps J_m}(1+|2^{-J_m}\circ x|)^{-C}\prod_{(l_1,l_2)\in Y_d}2^{-J_m(l_1+l_2)}.
\end{equation}
Using both \eqref{mc61} and \eqref{mc62}, it follows that
\begin{equation}\label{tj22.6}
\Big|\int_{\mathbb{R}^{|Y_d|}}N_m(x)\,dx\Big|\lesssim 2^{-J_m}.
\end{equation}

We would like to prove \eqref{tj23} using the Cotlar--Stein lemma. For this we need to modify the kernels $N_m$ to achieve a cancellation. More precisely, given a fixed fraction $a/q\in \mathcal{S}_\infty$ and $2^{J_m}\geq q^{(8d)^8}$ we would like to define kernels $N'_m:\G_0\to\mathbb{C}$ with the properties
\begin{equation}\label{tj30}
\begin{split}
&\sum_{h\in \G_0}N'_m(h)e^{2\pi i h\cdot a/q}e^{-2\pi i(v\cdot h)\cdot a/q}=\sum_{h\in \G_0}N'_m(h)e^{2\pi i h\cdot a/q}e^{-2\pi i(h\cdot v)\cdot a/q}=0,\qquad v\in\G_0,\\
&\|N_m-N'_m\|_{L^1(\G_0)}\lesssim 2^{-J_m/4},\\
&N'_m(h)=0\qquad\text{ if }\qquad h\notin \mathcal{D}_{2^{J_m(1+2\eps)}}.
\end{split}
\end{equation}
To prove this, we introduce a decomposition of elements in the group $\G_0$, adapted to the denominator $q$. Let
\begin{equation}\label{tj30.1}
\begin{split}
&\mathbb{H}_q=\{h\in \G_0:\,h=(qm_{l_1l_2})_{(l_1,l_2)\in Y_d},\,m_{l_1l_2}\in\Z\},\\
&R_q=\{b\in\G_0:\,b_{l_1l_2}\in[0,q-1]\cap\Z\},
\end{split}
\end{equation}
and notice that
\begin{equation}\label{tj30.2}
\text{ the map }(h,b)\to h\cdot b\text{ defines a bijection from }\mathbb{H}_q\times R_q\text{ to }\G_0.
\end{equation}
The cancellation condition in the first line or \eqref{tj30} holds provided that
\begin{equation}\label{tj30.3}
\sum_{h\in\mathbb{H}_q}N'_m(h\cdot b)=0\text{ for any }b\in R_q.
\end{equation}
Therefore we set, for any $h\in\G_0$
\begin{equation}\label{tj30.4}
\begin{split}
&N'_m(h)=N_m(h)-\widetilde{\eta}_{\leq J_m}(h)\prod_{(l_1,l_2)\in Y_d}2^{-J_m(l_1+l_2)}\sum_{b\in R_q}\gamma_b\mathbf{1}_{\mathbb{H}_q\cdot b}(h),\\
&\gamma_b=\big[\sum_{g\in\mathbb{H}_q}N_m(g\cdot b)\big]\big[\sum_{g\in\mathbb{H}_q}\widetilde{\eta}_{\leq J_m}(g\cdot b)\prod_{(l_1,l_2)\in Y_d}2^{-J_m(l_1+l_2)}\big]^{-1}\qquad\text{ for any }b\in\mathbb{R}_q.
\end{split}
\end{equation}
The support assertion in \eqref{tj30}  follows from the definition. The cancellation assertion in \eqref{tj30} follows from \eqref{tj30.3}. Finally, to prove that
$\|N_m-N'_m\|_{L^1(\G_0)}\lesssim 2^{-J_m/4}$ it suffices to prove that
\begin{equation*}
|\gamma_b|\lesssim 2^{-J_m/4}\qquad\text{ for any }b\in R_q.
\end{equation*}
Recalling that $2^{J_m}\geq q^{(8d)^8}$ and using the definition of $\gamma_b$, it remains to prove that
\begin{equation}\label{tj30.5}
\big|\sum_{g\in\mathbb{H}_q}N_m(g\cdot b)\big|\lesssim 2^{-J_m/3}\qquad\text{ for any }b\in R_q.
\end{equation}
Using \eqref{tj22},
\begin{equation*}
\sup_{b\in R_q}|N_m(g\cdot b)-N_m(g)|\lesssim 2^{-J_m/2}(1+|2^{-J_m}\circ x|)^{-(4d)^4}\prod_{(l_1,l_2)\in Y_d}2^{-J_m(l_1+l_2)}.
\end{equation*}
Moreover, using \eqref{tj22.4}, \eqref{tj22.6},
\begin{equation*}
\big|\sum_{g\in\mathbb{H}_q}N_m(g)\big|\lesssim 2^{-J_m/2}.
\end{equation*}
The bound \eqref{tj30.5} follows from the last two bounds, which completes the proof of \eqref{tj30}.

We turn now to the proof of \eqref{tj23}. Let
\begin{equation*}
T_mF(g)=\sum_{h\in\G_0}F(h^{-1}\cdot g)e^{2\pi ih\cdot a/q}N'_m(h).
\end{equation*}
For \eqref{tj23} it suffices to prove that
\begin{equation*}
\Big\|\sum_{2^{J_m}\geq q^{(8d)^8}}T_m\Big\|_{L^2\to L^2}\lesssim q^{(4d)^4}.
\end{equation*}
In view of the Cotlar--Stein lemma, it suffices to prove that, for some $\delta=\delta(d)>0$
\begin{equation}\label{tj31}
\|T_mT^\ast_{m'}\|_{L^2\to L^2}+\|T^\ast_mT_{m'}\|_{L^2\to L^2}\lesssim 2^{-\delta(m'-m)}q^{2(4d)^4},\qquad 2^{J_{m'}}\geq 2^{J_m}\geq q^{(8d)^8}.
\end{equation}

We first prove \eqref{tj31} when $m'\geq m+1$. Using the cancellation conditions in \eqref{tj30},
\begin{equation*}
\begin{split}
T_mT_{m'}^\ast F(g)&=\sum_{v\in\G_0}F(v\cdot g)\big[\sum_{h\in\G_0}N'_m(h)e^{2\pi i h\cdot a/q}\overline{N'_{m'}}(v\cdot h)e^{-2\pi i(v\cdot h)\cdot a/q}\big]\\
&=\sum_{v\in\G_0}F(v\cdot g)\big[\sum_{h\in\G_0}N'_m(h)e^{2\pi i h\cdot a/q}[\overline{N'_{m'}}(v\cdot h)-\overline{N'_{m'}}(v)]e^{-2\pi i(v\cdot h)\cdot a/q}\big],
\end{split}
\end{equation*}
and
\begin{equation*}
\begin{split}
T_m^\ast T_{m'} F(g)&=\sum_{v\in\G_0}F(v^{-1}\cdot g)\big[\sum_{h\in\G_0}\overline{N'_m}(h)e^{-2\pi i h\cdot a/q}N'_{m'}(h\cdot v)e^{2\pi i(h\cdot v)\cdot a/q}\big]\\
&=\sum_{v\in\G_0}F(v^{-1}\cdot g)\big[\sum_{h\in\G_0}\overline{N'_m}(h)e^{-2\pi i h\cdot a/q}[N'_{m'}(h\cdot v)-N'_{m'}(v)]e^{2\pi i(h\cdot v)\cdot a/q}\big].
\end{split}
\end{equation*}
Thus
\begin{equation}\label{tj32}
\begin{split}
&\|T_mT^\ast_{m'}\|_{L^2\to L^2}+\|T^\ast_mT_{m'}\|_{L^2\to L^2}\\
&\lesssim \Big\|\sum_{h\in\G_0}|N'_m(h)|\big[|N'_{m'}(h\cdot v)-N'_{m'}(v)|+|N'_{m'}(v\cdot h)-N'_{m'}(v)|\big]\Big\|_{L^1_v}\\
&\lesssim  \|N'_m\|_{L^1}\|N_{m'}-N'_{m'}\|_{L^1}\\
&+\|N'_m\|_{L^1}\sup_{h\in \mathcal{D}^\#_{2^{J_m(1+2\eps)}}}\big(\|N_{m'}(h\cdot v)-N_{m'}(v)\|_{L^1_v}+\|N_{m'}(v\cdot h)-N_{m'}(v)\|_{L^1_v}\big).
\end{split}
\end{equation}
Using the bounds \eqref{tj22.4} and \eqref{tj30} and the separation assumption $J_{m'}\geq 2J_m$,
\begin{equation}\label{tj32.5}
\begin{split}
&\|N_{m'}-N'_{m'}\|_{L^1}\lesssim 2^{-J_{m'}/4},\qquad \|N_{m}\|_{L^1}+\|N'_{m}\|_{L^1}\lesssim 2^{J_m/20},\\
&\sup_{h\in \mathcal{D}^\#_{2^{J_m(1+2\eps)}}}|N_{m'}(h\cdot v)-N_{m'}(v)|\lesssim \mathbf{1}_{\mathcal{D}_{2^{J_{m'}(1+8\eps)}}}(v)2^{J_m-J_{m'}}2^{8\eps J_{m'}}\prod_{(l_1,l_2)\in Y_d}2^{-J_{m'}(l_1+l_2)},\\
&\sup_{h\in \mathcal{D}^\#_{2^{J_m(1+2\eps)}}}|N_{m'}(v\cdot h)-N_{m'}(v)|\lesssim \mathbf{1}_{\mathcal{D}_{2^{J_{m'}(1+8\eps)}}}(v)2^{J_m-J_{m'}}2^{8\eps J_{m'}}\prod_{(l_1,l_2)\in Y_d}2^{-J_{m'}(l_1+l_2)}.
\end{split}
\end{equation}
Using \eqref{tj32} it follows that
\begin{equation}\label{tj33}
\|T_mT^\ast_{m'}\|_{L^2\to L^2}+\|T^\ast_mT_{m'}\|_{L^2\to L^2}\lesssim 2^{-J_{m'}/10},
\end{equation}
which clearly suffices to prove \eqref{tj31} in this case.

Finally, we prove \eqref{tj31} when $m'=m$, which is equivalent to
\begin{equation}\label{tj40}
\Big\|\sum_{h\in\G_0}F(h)e^{2\pi i(g\cdot h)\cdot a/q}N_m(g\cdot h)\Big\|_{L^2_g}\lesssim q^{(4d)^4}\|F\|_{L^2}.
\end{equation}
Using the decomposition \eqref{tj30.1}-\eqref{tj30.2}, it suffices to prove that
\begin{equation*}
\sum_{b\in R_q}\Big\|\sum_{h\in\mathbb{H}_q}F(h\cdot b)e^{2\pi i(g\cdot h\cdot b)\cdot a/q}N_m(g\cdot h\cdot b)\Big\|_{L^2_g}\lesssim q^{(4d)^4}\|F\|_{L^2}.
\end{equation*}
Since $e^{2\pi i(g\cdot h\cdot b)\cdot a/q}$ does not depend on $h\in \mathbb{H}_q$, this is equivalent to proving that
\begin{equation}\label{tj40.5}
\sum_{b\in R_q}\Big\|\sum_{h\in\mathbb{H}_q}F_b(h)N_m(g^{-1}\cdot h\cdot b)\Big\|_{L^2_g}\lesssim q^{(4d)^4}\big[\sum_{b\in R_q}\|F_b\|^2_{L^2(\mathbb{H}_q)}\big]^{1/2}.
\end{equation}
We notice that $R_q$ has $q^{|Y_d|}$ elements. Therefore, it suffices to prove that for any $F,G\in L^2(\G_0)$
\begin{equation}\label{tj41}
\Big|\sum_{h,g\in\G_0}F(h)N_m(h^{-1}\cdot g)G(g)\Big|\lesssim \|F\|_{L^2}\|G\|_{L^2}.
\end{equation}

We derive \eqref{tj41} as a consequence of $L^2$ boundedness of a singular Radon transform on the nilpotent Lie group $\G_0^\#$. Let
\begin{equation*}
\mathcal{C}=[0,1)^{|Y_d|}\subseteq\G_0^\#
\end{equation*}
and notice that
\begin{equation}\label{tj42}
\text{ the map }(g,\mu)\to g\cdot\mu\text{ defines a measure-preserving bijection from }\G_0\times\mathcal{C} \text{ to }\G_0^\#.
\end{equation}
For any function $f\in L^2(\G_0)$ let
\begin{equation*}
f^\#(g\cdot\mu)=f(g)\text{ for any }(g,\mu)\in\G_0\times\mathcal{C},\quad f^\#\in L^2(\G_0^\#),\quad \|f^\#\|_{L^2(\G_0^\#)}=\|f\|_{L^2(\G_0)}.
\end{equation*}
Then we write, for any $F,G\in L^2(\G_0)$
\begin{equation}\label{tj43}
\begin{split}
&\sum_{h,g\in\G_0}F(h)N_m(h^{-1}\cdot g)G(g)=\int_{\mathcal{C}\times\mathcal{C}}\sum_{h,g\in\G_0}F^\#(h\cdot\mu)N_m(h^{-1}\cdot g)G^\#(g\cdot\nu)\,d\mu d\nu\\
&=\int_{\G_0^\#\times\G_0^\#}F^\#(y)N_m(y^{-1}\cdot x)G^\#(x)\,dx dy\\
&+\int_{\mathcal{C}\times\mathcal{C}}\sum_{h,g\in\G_0}F^\#(h\cdot\mu)[N_m(h^{-1}\cdot g)-N_m(\mu^{-1}\cdot h^{-1}\cdot g\cdot\nu)]G^\#(g\cdot\nu)\,d\mu d\nu.
\end{split}
\end{equation}
Using \eqref{tj22.4}, we have
\begin{equation*}
\big\|\sup_{\mu,\nu\in\mathcal{C}}\big|N_m(x)-N_m(\mu^{-1}\cdot x\cdot\nu)\big|\,\big\|_{L^1_x(\G_0)}\lesssim 2^{-J_m/2}.
\end{equation*}
Thus
\begin{equation*}
\Big|\int_{\mathcal{C}\times\mathcal{C}}\sum_{h,g\in\G_0}F^\#(h\cdot\mu)[N_m(h^{-1}\cdot g)-N_m(\mu^{-1}\cdot h^{-1}\cdot g\cdot\nu)]G^\#(g\cdot\nu)\,d\mu d\nu\Big|\lesssim \|F\|_{L^2}\|G\|_{L^2}.
\end{equation*}
Using \eqref{tj43}, for \eqref{tj41} it suffices to prove that
\begin{equation}\label{tj44}
\Big|\int_{\G_0^\#\times\G_0^\#}F(y)N_m(y^{-1}\cdot x)G(x)\,dx dy\Big|\lesssim \|F\|_{L^2(\G_0^\#)}\|G\|_{L^2(\G_0^\#)}
\end{equation}
for any $F,G\in L^2(\G_0^\#)$.

We examine the formula \eqref{tj22} and define
\begin{equation*}
N''_m(x)=\prod_{(l_1,l_2)\in Y_d}2^{-J_m(l_1+l_2)}\int_{\mathbb{R}^{|Y_d|}}\sum_{j_1,k_1,\ldots,j_r,k_r\in[J_m(1-\kappa),J_m]\cap\Z}P_{j_1,k_1,\ldots,j_r,k_r}(\beta)e^{2\pi i (2^{-J_m}\circ x)\cdot\beta}\,d\beta.
\end{equation*}
Using \eqref{mc62}
\begin{equation*}
\|N_m-N''_m\|_{L^1(\G_0^\#)}\lesssim 1.
\end{equation*}
Therefore, for \eqref{tj44} it suffices to prove that for any $F\in C^\infty_0(\G_0^\#)$
\begin{equation}\label{tj45}
\Big\|\int_{\G_0^\#}F(y^{-1}\cdot x)N''_m(y)\,dy\Big\|_{L^2_x(\G_0^\#)}\lesssim \|F\|_{L^2(\G_0^\#)}.
\end{equation}
Recalling the definition \eqref{mc60} we notice that, for any $F\in C^\infty_0(\G_0^\#)$,
\begin{equation*}
\int_{\G_0^\#}F(y^{-1}\cdot x)N''_m(y)\,dy=\sum_{j_1,k_1,\ldots,j_r,k_r\in[J_m(1-\kappa),J_m]\cap\Z}[(H_{j_1}^\#)^\ast H_{k_1}^\#\ldots (H_{j_r}^\#)^\ast H_{k_r}^\#](F)(x)
\end{equation*}
where, by definition,
\begin{equation}\label{tj46}
H_j^\#f(x)=\int_{\mathbb{R}}K_j(t)f(A_0(t)^{-1}\cdot x)\,dt.
\end{equation}
Therefore, for \eqref{tj45} it suffices to prove that
\begin{equation}\label{tj47}
\big\|\sum_{j\in[J_m(1-\kappa),J_m]\cap\Z}H_{j}^\#\big\|_{L^2(\G_0^\#)\to L^2(\G_0^\#)}\lesssim 1.
\end{equation}

The bound \eqref{tj47} is essentially known, as a consequence of Theorem 3.4 in \cite{RiSt2}. We can also reprove it easily, using the bounds we have proved so far. As in the proof of Lemma \ref{separated2}, using the Cotlar--Stein lemma it suffices to prove that
\begin{equation}\label{tj48}
\|H_k^\#((H_j^\#)^\ast H_j^\#)^r\|_{L^2(\G_0^\#)\to L^2(\G_0^\#)}+\|(H_k^\#)^\ast(H_j^\#(H_j^\#)^\ast)^r\|_{L^2(\G_0^\#)\to L^2(\G_0^\#)}\lesssim 2^{-\delta(j-k)}
\end{equation}
for some $\delta=\delta(d)>0$ and any $k\leq j\in[J_m(1-\kappa),J_m]\cap\Z$. The operator $H_k^\#((H_j^\#)^\ast H_j^\#)^r$ is a convolution operator on the group $\G_0^\#$ defined by the kernel
\begin{equation*}
x\to\int_{\mathbb{R}}K_k(t)M''_j(A_0(t)^{-1}\cdot x)\,dt,
\end{equation*}
where
\begin{equation*}
M''_j(x)=\prod_{(l_1,l_2)\in Y_d}2^{-j(l_1+l_2)}\int_{\mathbb{R}^{|Y_d|}}P_{j,j,\ldots,j,j}(\beta)e^{2\pi i (2^{-j}\circ x)\cdot\beta}\,d\beta.
\end{equation*}
Using \eqref{mc62} and integration by parts, the kernels $M''_j$ satisfy the same bounds as the kernels $M_j$ defined in \eqref{tj4}, namely
\begin{equation*}
|M''_j(x)|+\sum_{(l_1,l_2)\in Y_d}2^{j(l_1+l_2)}|\partial_{x_{l_1,l_2}}M''_j(x)|\lesssim(1+|2^{-j}\circ x|)^{-(4d)^4}\prod_{(l_1,l_2)\in Y_d}2^{-j(l_1+l_2)}.
\end{equation*}
Using the cancellation assumption $\int_{\mathbb{R}}K_k(t)\,dt=0$ in \eqref{mc2}, it follows that the $L^1(\G_0^\#)$ of the kernel of the operator $H_k^\#((H_j^\#)^\ast H_j^\#)^r$ is $\lesssim 2^{k-j}$, which suffices to prove the desired bound on the first term in the left-hand side of \eqref{tj48}. The bound on the second term is similar. This completes the proof of the lemma.
\end{proof}

Finally we verify the main inequalities in \eqref{na99}. Proposition \ref{separated} follows from Lemma \ref{prop1}, Lemma \ref{separated2} , and Lemma \ref{separated5} below. This completes the proof of Theorem \ref{Main2}.

\begin{lemma}\label{separated5}
Assume $J_1,\ldots,J_K\in[C(d),\infty)$ satisfy the separation condition
\begin{equation}\label{tj50}
J_{m+1}\geq 2J_{m},\qquad m=1,\ldots,K-1.
\end{equation}
For $m=1,\ldots,K$ let, as before,
\begin{equation*}
S_m=\sum_{j\in[J_m(1-\kappa),J_m]\cap\Z}H_j.
\end{equation*}
Then, for some $\delta=\delta(d)>0$ and any $m=1,\ldots,K-1$
\begin{equation}\label{tj51}
\begin{split}
\big\|S_m&[(S_{m+1}^\ast S_{m+1})^r+\ldots+(S_K^\ast S_K)^r]\big\|_{L^2\to L^2}\\
&+\big\|S_m^\ast[(S_{m+1} S_{m+1}^\ast)^r+\ldots+(S_K S_K^\ast)^r]\big\|_{L^2\to L^2}\lesssim 2^{-\delta m}.
\end{split}
\end{equation}
\end{lemma}

\begin{proof}[Proof of Lemma \ref{separated5}] As before, we focus on the bound on the first term in \eqref{tj51}. We already know from Lemma \ref{separated3} that
\begin{equation*}
\big\|(S_{m+1}^\ast S_{m+1})^r+\ldots+(S_K^\ast S_K)^r\big\|_{L^2\to L^2}\lesssim 1,\qquad m=1,\ldots,K-1,
\end{equation*}
so it remains to prove that composition with the operator $S_m$ contributes an additional factor of $2^{-\delta m}$.

We fix $m$ and apply Proposition \ref{majarcs} to the operators $(S_n^\ast S_n)^r$, $n=m+1,\ldots,K$. The contribution of the error terms is clearly acceptable. For $n=m+1,\ldots,K$ and $a/q\in \mathcal{S}_{2^{3d^2\eps J_n}}$ let
\begin{equation}\label{tj52}
U_n^{a/q}F(g)=\sum_{h\in\G_0}F(h^{-1}\cdot g)2^{2\pi i h\cdot a/q}N_n(h),
\end{equation}
where $N_n$ are the kernels defined in \eqref{tj22}. After rearranging the sum, for \eqref{tj51} it suffices to prove that
\begin{equation*}
\sum_{a/q\in \mathcal{S}_\infty}S(a/q)\Big\|S_m\sum_{n\in[m+1,K]\cap\Z,\,2^{3d^2\eps J_n}\geq q}U_n^{a/q}\Big\|_{L^2\to L^2}\lesssim 2^{-\delta m}.
\end{equation*}
We already know, see \eqref{tj23}, that
\begin{equation*}
\Big\|S_m\sum_{n\in[m+1,K]\cap\Z,\,2^{3d^2\eps J_n}\geq q}U_n^{a/q}\Big\|_{L^2\to L^2}\lesssim q^{(4d)^4}.
\end{equation*}
In view of the rapid decay of the coefficients $S(a/q)$, see Lemma \ref{Saq}, it only remains to estimate the contribution of fractions $a/q$ with denominators $q$ small relative to $2^{J_m}$; more precisely, it remains to prove that for any $m\in[1,K-1]\cap \Z$ and any $a/q\in \mathcal{S}_{2^{\eps J_m}}$
\begin{equation}\label{tj53}
\Big\|S_m\sum_{n\in[m+1,K]\cap\Z}U_n^{a/q}\Big\|_{L^2\to L^2}\lesssim 2^{-\delta m}q^{(4d)^4}.
\end{equation}

The kernel of the operator $S_mU_n^{a/q}$, $n\geq m+1$, is
\begin{equation*}
g\to\sum_{t\in\Z}\sum_{j\in[J_m(1-\kappa),J_m]}K_j(t) e^{2\pi i(A_0(t)^{-1}\cdot g)\cdot a/q}N_n(A_0(t)^{-1}\cdot g)
\end{equation*}
which we write as
\begin{equation*}
g\to Z_m(g)N_n(g)+\sum_{t\in\Z}\sum_{j\in[J_m(1-\kappa),J_m]}K_j(t) e^{2\pi i(A_0(t)^{-1}\cdot g)\cdot a/q}[N_n(A_0(t)^{-1}\cdot g)-N_n(g)]
\end{equation*}
where
\begin{equation}\label{tj54}
Z_m(g)=\sum_{t\in\Z}\sum_{j\in[J_m(1-\kappa),J_m]}K_j(t) e^{2\pi i(A_0(t)^{-1}\cdot g)\cdot a/q}.
\end{equation}
It follows from \eqref{tj22.4} and the separation condition \eqref{tj50} that
\begin{equation*}
\Big\|\sum_{t\in\Z}\sum_{j\in[J_m(1-\kappa),J_m]}K_j(t) e^{2\pi i(A_0(t)^{-1}\cdot g)\cdot a/q}[N_n(A_0(t)^{-1}\cdot g)-N_n(g)]\Big\|_{L^1_g(\G_0)}\lesssim 2^{-J_n/4}
\end{equation*}
Therefore, for \eqref{tj53} it remains to prove that for any $m\in[1,K-1]\cap \Z$ and any $a/q\in \mathcal{S}_{2^{\eps J_m}}$
\begin{equation}\label{tj55}
\Big\|\sum_{h\in G_0}F(h^{-1}\cdot g)\sum_{n\in[m+1,K]\cap\Z}N_n(h)Z_m(h)\Big\|_{L^2(\G_0)}\lesssim 2^{-\delta m}q^{(4d)^4}\|F\|_{L^2(\G_0)}.
\end{equation}

We examine now the functions $Z_m:\G_0\to\mathbb{C}$ defined in \eqref{tj54}. Clearly,
\begin{equation}\label{tj56}
Z_m(g_1\cdot h\cdot g_2)=Z_m(g_1\cdot g_2)\qquad\text{ for any }g_1.g_2\in\G_0\text{ and }h\in \mathbb{H}_q,
\end{equation}
where the subgroup $\mathbb{H}_q$ is defined in \eqref{tj30.1}. Moreover, for any $g\in\G_0$,
\begin{equation*}
\begin{split}
|Z_m(g)|&\leq \sum_{y\in Z_q}\Big|\sum_{x\in\Z}\sum_{j\in[J_m(1-\kappa),J_m]}K_j(qx+y) e^{2\pi i(A_0(qx+y)^{-1}\cdot g)\cdot a/q}\Big|\\
&\leq \sum_{y\in Z_q}\sum_{j\in[J_m(1-\kappa),J_m]}\Big|\sum_{x\in\Z}K_j(qx+y)\Big|.
\end{split}
\end{equation*}
It follows from \eqref{mc2} and the assumption $q\leq 2^{\eps J_m}$ that
\begin{equation}\label{tj57}
\sup_{g\in\G_0}|Z_m(g)|\lesssim 2^{-J_m/2}.
\end{equation}

We turn now to the proof of \eqref{tj55}, which is similar to the proof of \eqref{tj23}. The functions $Z_m$ replace the oscillatory factors $h\to e^{2\pi ih\cdot a/q}$; these functions satisfy the identities \eqref{tj56} and the estimates \eqref{tj57}, which provide the additional exponential decay in $m$.  We define the kernels $N'_n$ as in \eqref{tj30.4} and the operators
\begin{equation*}
V_nF(g)=\sum_{h\in\G_0}F(h^{-1}\cdot g)N'_n(h)Z_m(h).
\end{equation*}
In view of the Cotlar--Stein lemma, it suffices to prove that for any $n'\geq n\geq m+1$
\begin{equation}\label{tj58}
\|V_nV_{n'}^\ast\|_{L^2\to L^2}+\|V_n^\ast V_{n'}\|_{L^2\to L^2}\lesssim 2^{-(n'-n)/100}2^{-J_m/100}.
\end{equation}

Using \eqref{tj30.3} and \eqref{tj56}, for any $h\in\G_0$ and $k\in[m+1,K]\cap\Z$
\begin{equation*}
\sum_{x\in\G_0}N'_k(x)Z_m(x)\overline{Z_m}(h\cdot x)=\sum_{x\in\G_0}N'_k(x)Z_m(x)\overline{Z_m}(x\cdot h)=0.
\end{equation*}
Therefore, assuming first that $n'\geq n+1$ in \eqref{tj58}, we write
\begin{equation*}
\begin{split}
&(V_nV_{n'}^\ast)F(g)=\sum_{h\in\G_0}F(h\cdot g)\Big[\sum_{x\in\G_0}N'_n(x)Z_m(x)\overline{Z_m}(h\cdot x)[\overline{N'_{n'}}(h\cdot x)-\overline{N'_{n'}}(h)]\Big],\\
&(V_n^\ast V_{n'})F(g)=\sum_{h\in\G_0}F(h^{-1}\cdot g)\Big[\sum_{x\in\G_0}\overline{N'_n}(x)\overline{Z_m}(x)Z_m(x\cdot h)[N'_{n'}(x\cdot h)-N'_{n'}(h)]\Big].
\end{split}
\end{equation*}
Therefore, using \eqref{tj57},
\begin{equation*}
\begin{split}
\|V_n&V_{n'}^\ast\|_{L^2\to L^2}+\|V_n^\ast V_{n'}\|_{L^2\to L^2}\\
&\lesssim 2^{-J_m}\Big\|\sum_{x\in\G_0}|N'_n(x)|\,[|N'_{n'}(h\cdot x)-N'_{n'}(h)|+|N'_{n'}(x\cdot h)-N'_{n'}(h)|]\Big\|_{L^1_h(\G_0)},
\end{split}
\end{equation*}
and the desired bound \eqref{tj58} follows from \eqref{tj32} and \eqref{tj32.5} in this case.

Finally, to prove \eqref{tj58} when $n=n'$, it suffices to prove that
\begin{equation*}
\Big\|\sum_{h\in\G_0}F(h)N'_n(g\cdot h)Z_m(g\cdot h)\Big\|_{L^2_g}\lesssim 2^{-J_m/20}\|F\|_{L^2},
\end{equation*}
for any $F\in L^2(\G_0)$ and $n\geq m+1$. Using the decomposition \eqref{tj30.1}-\eqref{tj30.2}, it suffices to prove that
\begin{equation*}
\sum_{b\in R_q}\Big\|\sum_{x\in\mathbb{H}_q}F(x\cdot b)N'_n(g\cdot x\cdot b)Z_m(g\cdot x\cdot b)\Big\|_{L^2_g}\lesssim 2^{-J_m/20}\|F\|_{L^2}.
\end{equation*}
Using \eqref{tj56}-\eqref{tj57}, it suffices to prove that for any functions $F_b\in L^2(\mathbb{H}_q)$, $b\in R_q$,
\begin{equation*}
\sum_{b\in R_q}\Big\|\sum_{x\in\mathbb{H}_q}F_b(x)N'_n(g\cdot x\cdot b)\Big\|_{L^2_g}\lesssim 2^{J_m/4}\big[\sum_{b\in R_q}\|F_b\|_{L^2(\mathbb{H}_q)}^2\big]^{1/2}.
\end{equation*}
This bound was already proved in Lemma \ref{separated3}, see \eqref{tj40.5}.
\end{proof}

\section{Estimates on oscillatory sums and oscillatory integrals}\label{oscil}

With the notation in section \ref{transference}, for $r\geq 1$ let $D,\widetilde{D}:\mathbb{R}^r\times\mathbb{R}^r\to\G_0^\#$,
\begin{equation}\label{pro0.3}
\begin{split}
&D((n_1,\ldots,n_r),(m_1,\ldots,m_r))=A_0(n_1)^{-1}\cdot A_0(m_1)\cdot\ldots\cdot A_0(n_r)^{-1}\cdot A_0(m_r),\\
&\widetilde{D}((n_1,\ldots,n_r),(m_1,\ldots,m_r))=A_0(n_1)\cdot A_0(m_1)^{-1}\cdot\ldots\cdot A_0(n_r)\cdot A_0(m_r)^{-1},
\end{split}
\end{equation}
By definition, we have
\begin{equation*}
[A_0(n)]_{l_1l_2}=\begin{cases}
n^{l_1}&\text{ if }l_2=0,\\
0&\text { if }l_2\geq 1,
\end{cases}
\qquad
[A_0(n)^{-1}]_{l_1l_2}=\begin{cases}
-n^{l_1}&\text{ if }l_2=0,\\
n^{l_1+l_2}&\text { if }l_2\geq 1.
\end{cases}
\end{equation*}
Thus, for $x=(x_1,\ldots,x_r)\in\mathbb{R}^r$ and $y=(y_1,\ldots,y_r)\in\mathbb{R}^r$
\begin{equation}\label{pro0.4}
[D(x,y)]_{l_1l_2}=\begin{cases}
\sum\limits_{j=1}^r(y_j^{l_1}-x_j^{l_1})&\text{ if }l_2=0,\\
\sum\limits_{1\leq j_1<j_2\leq r}(y_{j_1}^{l_1}-x_{j_1}^{l_1})(y_{j_2}^{l_2}-x_{j_2}^{l_2})+\sum\limits_{j=1}^r(x_j^{l_1+l_2}-x_j^{l_1}y_j^{l_2})&\text{ if }l_2\geq 1,
\end{cases}
\end{equation}
and
\begin{equation}\label{pro0.5}
[\widetilde{D}(x,y)]_{l_1l_2}=\begin{cases}
\sum\limits_{j=1}^r(x_j^{l_1}-y_j^{l_1})&\text{ if }l_2=0,\\
\sum\limits_{1\leq j_1<j_2\leq r}(x_{j_1}^{l_1}-y_{j_1}^{l_1})(x_{j_2}^{l_2}-y_{j_2}^{l_2})+\sum\limits_{j=1}^r(y_j^{l_1+l_2}-x_j^{l_1}y_j^{l_2})&\text{ if }l_2\geq 1.
\end{cases}
\end{equation}

The multi-variable polynomials $D$ and $\widetilde{D}$ appear when we consider high powers of our singular integral operators, see for example the formula \eqref{mc7.5}. In this section we prove two estimates on certain oscillatory sums and integrals involving these polynomials.

For integers $P\geq 1$ assume $\phi_P^{(j)},\psi_P^{(j)}:\mathbb{R}\to\mathbb{R}$, $j=1,\ldots,r$, are $C^1$ functions with the properties
\begin{equation}\label{pro0.1}
\sup_{j=1,\ldots,r}[|\phi_P^{(j)}|+|\psi_P^{(j)}|]\leq \mathbf{1}_{[-P.P]},\qquad \sup_{j=1,\ldots,r}\int_{\mathbb{R}}|[\phi^{(j)}_P]'(x)|+|[\psi^{(j)}_P]'(x)|\,dx\leq 1.
\end{equation}
For $\theta=(\theta_{l_1l_2})_{(l_1,l_2)\in Y_d}\in\mathbb{R}^{|Y_d|}$, $r\geq 1$, and $P\geq 1$ let
\begin{equation*}
S_{P,r}(\theta)=\sum_{n,m\in\Z^r}e^{-2\pi i D((n_1,\ldots,n_r),(m_1,\ldots,m_r))\cdot\theta}\phi_P^{(1)}(n_1)\ldots \phi_P^{(r)}(n_r)\psi_P^{(1)}(m_1)\ldots \psi_P^{(r)}(m_r)
\end{equation*}
and
\begin{equation*}
\widetilde{S}_{P,r}(\theta)=\sum_{n,m\in\Z^r}e^{-2\pi i \widetilde{D}((n_1,\ldots,n_r),(m_1,\ldots,m_r))\cdot\theta}\phi_P^{(1)}(n_1)\ldots \phi_P^{(r)}(n_r)\psi_P^{(1)}(m_1)\ldots \psi_P^{(r)}(m_r).
\end{equation*}

\begin{proposition}\label{minarcs} There is a constant $\overline{C}=\overline{C}(d)$ sufficiently large such that for all $r\geq 1$ and all $\eps\in(0,1/2]$
\begin{equation}\label{pro0.2}
|S_{P,r}(\theta)|+|\widetilde{S}_{P,r}(\theta)|\lesssim_rP^{2r}P^{\overline{C}-r\eps/\overline{C}},\qquad P=1,2,\ldots,
\end{equation}
provided that there is a pair $(l_1,l_2)\in Y_d$ and an irreducible fraction $a/q\in\mathbb{Q}$, $q\in\Z_+^\ast$, such that
\begin{equation*}
|\theta_{l_1l_2}-a/q|\leq 1/q^2\text{ and }q\in[P^\eps,P^{l_1+l_2-\eps}].
\end{equation*}
\end{proposition}

To prove Proposition \ref{minarcs} we use a variant of the Weyl method, as in \cite{Da1} and \cite{Bi}. We provide all the details, for the sake of self-containedness, with the exception of the following key lemma, see Lemma 3.3 in \cite{Da1}:

\begin{lemma}\label{Dav}
Assume that $L_1,\ldots,L_n:\mathbb{R}^n\to\mathbb{R}$ are $n$ linear forms, $L_j(u)=\sum_{k=1}^n\lambda_{jk}u_k$, satisfying the symmetry condition
\begin{equation}\label{pro40}
\lambda_{jk}=\lambda_{kj},\,\,j,k=1,\ldots,n.
\end{equation}
Assume that $A>1$, $Z\in(0,1]$, and let $U(Z)$ denote the number of points $u\in\Z^n$ satisfying
\begin{equation*}
|u|\leq ZA,\qquad\sup_{j\in\{1,\ldots,n\}}\|L_j(u)\|\leq ZA^{-1},
\end{equation*}
where $\|y\|$ denotes the distance from $y$ to $\Z$ for any $y\in\mathbb{R}$. Then, for any $0<Z_1\leq Z_2\leq 1$,
\begin{equation*}
U(Z_2)\lesssim_n(Z_2/Z_1)^nU(Z_1).
\end{equation*}
\end{lemma}

\begin{proof}[Proof of Proposition \ref{minarcs}] We will only prove the estimate for $|S_{P,r}(\theta)|$; the estimate for $|\widetilde{S}_{P,r}(\theta)|$ follows by a very similar argument. It follows from \eqref{pro0.1} that $|S_{P,r}(\theta)|\lesssim_rP^{2r}$. Therefore, in proving \eqref{pro0.2} we may assume that $P\geq C_r$ and $r\geq\overline{C}^2/\eps$. We divide the proof in several steps.

{\bf{Step 1.}} For $n=(n_1,\ldots,n_r)$ fixed, let
\begin{equation*}
\begin{split}
&D^0(m)=D^0(m_1,\ldots,m_r)=D((n_1,\ldots,n_r),(m_1,\ldots,m_r))\in\Z^{|Y_d|},\\
&\Psi^0_P(m)=\psi_P^{(1)}(m_1)\ldots \psi_P^{(r)}(m_r).
\end{split}
\end{equation*}
It suffices to prove that for any $n=(n_1,\ldots,n_r)\in\Z^{r}$ fixed, with $|n_j|\leq P$,
\begin{equation}\label{pro0.8}
|S_{P,r}^n(\theta)|\lesssim_rP^{r}P^{\overline{C}-r\eps/\overline{C}}
\end{equation}
where
\begin{equation}\label{pro0}
S_{P,r}^n(\theta)=\sum_{w\in\Z^{r}}e^{-2\pi i D^0(w)\cdot\theta}\Psi^0_P(w).
\end{equation}
In addition, in view of \eqref{pro0.4},
\begin{equation}\label{pro0.9}
D(w)_{l_1l_2}\text{ is a polynomial of degree }l_1+l_2\text{ in }w\text{ for any }(l_1,l_2)\in Y_d.
\end{equation}

We fix a sequence $0<\delta_{2d-1}<\ldots<\delta_1<\eps$,
\begin{equation}\label{pro1.5}
\delta_l=\eps/C_0^l,\qquad C_0=C_0(d)\gg 1.
\end{equation}
Using Dirichlet's lemma, for any $(l_1,l_2)\in Y_d$ one can fix approximations
\begin{equation}\label{pro1}
\begin{split}
&\theta_{l_1l_2}=\frac{a_{l_1l_2}}{q_{l_1l_2}}+\beta_{l_1l_2},\,a_{l_1l_2},q_{l_1l_2}\in\Z,\\
&(a_{l_1l_2},q_{l_1l_2})=1,\,1\leq q_{l_1l_2}\leq P^{l_1+l_2-\delta_{l_1+l_2}},\,|\beta_{l_1l_2}|\leq (q_{l_1l_2}P^{l_1+l_2-\delta_{l_1+l_2}})^{-1}.
\end{split}
\end{equation}
In view of the hypothesis, there is $d_0\in\{1,\ldots,2d-1\}$ such that
\begin{equation}\label{pro2}
q_{l_1l_2}\leq P^{\delta_{l_1+l_2}}\text{ if }l_1+l_2\geq d_0+1\text{ and }q_{l_1l_2}\geq P^{\delta_{d_0}}\text{ for some }l_1,l_2\text{ with }l_1+l_2=d_0.
\end{equation}

Let
\begin{equation*}
\begin{split}
&D^{l}(w;v^{(1)},\ldots,v^{(l)})=D^{l-1}(w+v^{(l)};v^{(1)},\ldots,v^{(l-1)})-D^{l-1}(w;v^{(1)},\ldots,v^{(l-1)}),\\
&\Psi_P^l(w;v^{(1)},\ldots v^{(l)})=\Psi_P^{l-1}(w+v^{(l)};v^{(1)},\ldots v^{(l-1)})\Psi_P^{l-1}(w;v^{(1)},\ldots v^{(l-1)}),
\end{split}
\end{equation*}
for $l=1,2,\ldots$. Using the formula \eqref{pro0},
\begin{equation*}
\begin{split}
|S^n_{P,r}(\theta)|^2&\leq\sum_{v^{(1)}\in\Z^{r}}\Big|\sum_{w\in\Z^{r}}e^{-2\pi i(D^0(w+v^{(1)})-D^0(w))\cdot\theta}\Psi^0_P(w+v^{(1)})\Psi^0_P(w)\Big|\\
&\leq\sum_{v^{(1)}\in\Z^{r}}\Big|\sum_{w\in\Z^{r}}e^{-2\pi iD^1(w;v^{(1)})\cdot\theta}\Psi^1_P(w;v^{(1)})\Big|.
\end{split}
\end{equation*}
We repeat this estimate $d_0-1$ times\footnote{If $d_0=1$ then the formula \eqref{pro0} gives already the estimate \eqref{pro6}.}. Using the Cauchy inequality, it follows that
\begin{equation}\label{pro6}
\begin{split}
&|S^n_{P,r}(\theta)|^{2^{d_0-1}}P^{-r(2^{d_0-1}-d_0)}\\
&\lesssim _r\sum_{|v^{(1)}|+\ldots+|v^{(d_0-1)}|\lesssim_rP}\Big|\sum_{w\in\Z^{r}}e^{-2\pi iD^{d_0-1}(w;v^{(1)},\ldots,v^{(d_0-1)})\cdot\theta}\Psi^{d_0-1}_P(w;v^{(1)},\ldots,v^{(d_0-1)})\Big|.
\end{split}
\end{equation}

It follows from \eqref{pro0.9} that $[D^{d_0-1}(w;v^{(1)},\ldots,v^{(d_0-1)})]_{l_1l_2}$ is a polynomial of degree at most $l_1+l_2-d_0+1$ in $w$, for any $v^{(1)},\ldots,v^{(d_0-1)}\in\Z^{r}$ fixed. Let
\begin{equation*}
Q=\prod_{l_1+l_2\geq d_0+1}q_{l_1l_2},
\end{equation*}
see \eqref{pro1}. In view of the assumption \eqref{pro2},
\begin{equation*}
1\leq Q\leq P^{2d^2\delta_{d_0+1}},
\end{equation*}
and we estimate, for any $v^{(1)},\ldots,v^{(d_0-1)}\in\Z^{r}$ fixed,
\begin{equation}\label{pro20}
\begin{split}
&\Big|\sum_{w\in\Z^{r}}e^{-2\pi iD^{d_0-1}(w;v^{(1)},\ldots,v^{(d_0-1)})\cdot\theta}\Psi^{d_0-1}_P(w;v^{(1)},\ldots,v^{(d_0-1)})\Big|\\
&=\Big|\sum_{w\in\Z^{r}}e^{-2\pi i\sum_{l_1+l_2=d_0}D^{d_0-1}(w;v^{(1)},\ldots,v^{(d_0-1)})_{l_1l_2}\cdot\theta_{l_1l_2}}A(w)\Big|\\
&\leq\sum_{y\in Z_Q^{r}}\Big|\sum_{x\in\Z^{r}}e^{-2\pi i\sum_{l_1+l_2=d_0}D^{d_0-1}(x;v^{(1)},\ldots,v^{(d_0-1)})_{l_1l_2}\cdot Q\theta_{l_1l_2}}A(Qx+y)\Big|,
\end{split}
\end{equation}
where
\begin{equation*}
A(w)=e^{-2\pi i\sum_{l_1+l_2\geq d_0+1}D^{d_0-1}(w;v^{(1)},\ldots,v^{(d_0-1)})_{l_1l_2}\cdot\theta_{l_1l_2}}\Psi^{d_0-1}_P(w;v^{(1)},\ldots,v^{(d_0-1)}).
\end{equation*}
We examine now the function $A'(x)=A(Qx+y)$, $y\in Z_Q^{r}$ fixed. Using \eqref{pro1},
\begin{equation*}
\begin{split}
A'(x)=&A''(y,v^{(1)},\ldots,v^{(d_0-1)})\\
&\times\Psi^{d_0-1}_P(Qx+y;v^{(1)},\ldots,v^{(d_0-1)})e^{-2\pi i\sum_{l_1+l_2\geq d_0+1}D^{d_0-1}(Qx+y;v^{(1)},\ldots,v^{(d_0-1)})_{l_1l_2}\cdot\beta_{l_1l_2}},
\end{split}
\end{equation*}
where $x\in\Z^{r}$ and $|A''(y,v^{(1)},\ldots,v^{(d_0-1)})|=1$ . By definition, see also \eqref{pro0.4}, it is easy to see that $D^{d_0-1}(w;v^{(1)},\ldots,v^{(d_0-1)})_{l_1l_2}$ is a polynomial of degree at most $l_1+l_2$ in $w,n,v^{(1)},\ldots v^{(d_0-1)}$ with coefficients $\lesssim _r1$. Since $|\beta_{l_1l_2}|\leq P^{-l_1-l_2+\delta_{d_0+1}}$ and $1\leq Q\leq P^{2d^2\delta_{d_0+1}}$,
\begin{equation*}
\sup_{|x|\lesssim_rP}\Big|\partial_{x_1}^{\sigma_1}\ldots\partial_{x_{r}}^{\sigma_{r}}e^{-2\pi i\sum_{l_1+l_2\geq d_0+1}D^{d_0-1}(Qx+y;v^{(1)},\ldots,v^{(d_0-1)})_{l_1l_2}\cdot\beta_{l_1l_2}}\Big|\lesssim _rP^{(-1+4d^2\delta_{d_0+1})(\sigma_1+\ldots+\sigma_{2r})}
\end{equation*}
for all $y\in Z_Q^{r}$, all $n,v^{(1)},\ldots,v^{(d_0-1)}\in\Z^{r}$ with $|n|+|v^{(1)}|+\ldots+|v^{(d_0-1)}|\lesssim_rP$, and $\sigma_1,\ldots,\sigma_{r}\in\{0,1\}$. Therefore, by summation by parts, it follows from \eqref{pro20} that
\begin{equation*}
\begin{split}
&\Big|\sum_{w\in\Z^{r}}e^{-2\pi iD^{d_0-1}(w;v^{(1)},\ldots,v^{(d_0-1)})\cdot\theta}\Psi^{d_0-1}_P(w;v^{(1)},\ldots,v^{(d_0-1)})\Big|\\
&\lesssim _rP^{20rd^2\delta_{d_0+1}}\sup_{a_j,b_j\in[-2P,2P]}\Big|\sum_{x_j\in[a_j,b_j]\cap\Z}e^{-2\pi i\sum_{l_1+l_2=d_0}D^{d_0-1}(x;v^{(1)},\ldots,v^{(d_0-1)})_{l_1l_2}\cdot Q\theta_{l_1l_2}}\Big|\\
&\lesssim_rP^{20rd^2\delta_{d_0+1}}\prod_{j=1}^{r}\min(P,\|B_j(v^{(1)},\ldots,v^{(d_0-1)})\|^{-1}),
\end{split}
\end{equation*}
where
\begin{equation}\label{pro21}
B_j(v^{(1)},\ldots,v^{(d_0-1)})=\frac{d}{dx_j}\Big[\sum_{l_1+l_2=d_0}D^{d_0-1}(x;v^{(1)},\ldots,v^{(d_0-1)})_{l_1l_2}\cdot Q\theta_{l_1l_2}\Big].
\end{equation}
In view of \eqref{pro6}, it remains to prove that
\begin{equation}\label{pro22}
\sum_{|v^{(1)}|+\ldots+|v^{(d_0-1)}|\leq P}\prod_{j=1}^{r}\min(P,\|B_j(v^{(1)},\ldots,v^{(d_0-1)})\|^{-1})\lesssim_rP^{rd_0}P^{\overline{C}}P^{-40rd^2\delta_{d_0+1}},
\end{equation}
assuming that $P^{\delta_{d_0}}\leq q_{l_1l_2}\leq P^{l_1+l_2-\delta_{d_0}}$ for some $(l_1,l_2)\in Y_d$ with $l_1+l_2=d_0$, see \eqref{pro2}.

For later use, we provide below a description of the functions $B_j$, $j=1,\ldots,r$. Assuming that $l_1+l_2=d_0$ and
\begin{equation}\label{pro401}
D(w)_{l_1l_2}=\sum_{j_1,\ldots,j_{d_0}=1}^{r}\lambda^{l_1l_2}_{j_1\ldots j_{d_0}}w_{j_1}\cdot\ldots\cdot w_{j_{d_0}}
\end{equation}
for some real-valued coefficients $\lambda^{l_1l_2}_{j_1\ldots j_{d_0}}$ satisfying the symmetry condition
\begin{equation}\label{pro402}
\lambda^{l_1l_2}_{j_1\ldots j_{d_0}}=\lambda^{l_1l_2}_{j_{\sigma(1)}\ldots j_{\sigma(d_0)}}\quad\text { for any permutation }\sigma\text{ of the set }\{1,\ldots,d_0\},
\end{equation}
it follows from the definition that
\begin{equation}\label{pro400}
B_j(v^{(1)},\ldots,v^{(d_0-1)})=d_0!\sum_{l_1+l_2=d_0}Q\theta_{l_1l_2}\sum_{j_1,\ldots,j_{d_0-1}=1}^{r}\lambda^{l_1l_2}_{j_1\ldots j_{d_0-1}j}v^{(1)}_{j_1}\cdot\ldots\cdot v^{(d_0-1)}_{j_{d_0-1}}.
\end{equation}

The claim \eqref{pro22} is easy to verify if $(l_1,l_2)=(1,0)$, using directly the definition \eqref{pro0.3}. Therefore, we will assume from now on that $2\leq d_0\leq 2d-1$.
\medskip

{\bf{Step 2.}} We show now that it suffices to prove that
\begin{equation}\label{pro30}
\begin{split}
\big|\{v^{(1)},\ldots,v^{(d_0-1)}\in B_{\Z^{r}}(P):\sup_{j=1,\ldots,r}\|&B_j(v^{(1)},\ldots,v^{(d_0-1)})\|\leq P^{-1}\}\big|\\
&\lesssim _rP^{r(d_0-1)}P^{\overline{C}}P^{-80rd^2\delta_{d_0+1}},
\end{split}
\end{equation}
where, by definition, $B_{\Z^m}(R)=\{v\in\Z^m:|v|\leq R\}$. Indeed, assuming \eqref{pro30}, it follows that
\begin{equation*}
\sum_{v^{(2)},\ldots,v^{(d_0-1)}\in B_{\Z^{r}}(P)}N_1(v^{(2)},\ldots,v^{(d_0-1)})\lesssim_rP^{r(d_0-1)}P^{\overline{C}}P^{-80rd^2\delta_{d_0+1}},
\end{equation*}
where, for any $v^{(2)},\ldots,v^{(d_0-1)}\in B_{\Z^{r}}(P)$,
\begin{equation*}
N_1(v^{(2)},\ldots,v^{(d_0-1)})=\big|\{v^{(1)}\in B_{\Z^{r}}(P):\|B_j(v^{(1)},\ldots,v^{(d_0-1)})\|\leq P^{-1},\,\,j=1,\ldots,r\}\big|.
\end{equation*}
On the other hand, arguing as in \cite[Lemma 3.2]{Da1}, for any $v^{(2)},\ldots,v^{(d_0-1)}\in B_{\Z^{r}}(P)$
\begin{equation*}
\sum_{v^{(1)}\in B_{\Z^{r}}(P)}\prod_{j=1}^{r}\min(P,\|B_j(v^{(1)},\ldots,v^{(d_0-1)})\|^{-1})\lesssim_r N_1(v^{(2)},\ldots,v^{(d_0-1)})(P\log P)^{r}.
\end{equation*}
The desired bound \eqref{pro22} follows from these two estimates.
\medskip

{\bf{Step 3.}} Let
\begin{equation*}
\rho=(\delta_{d_0+1}\delta_{d_0})^{1/2},\,\,\,\delta_{d_0+1}\ll\rho\ll\delta_{d_0}.
\end{equation*}
We show now that it suffices to prove that
\begin{equation}\label{pro32}
\begin{split}
\big|\{v^{(1)},\ldots,v^{(d_0-1)}\in B_{\Z^{r}}(P^\rho):\sup_{j=1,\ldots,r}\|&B_j(v^{(1)},\ldots,v^{(d_0-1)})\|\leq P^{-d_0+(d_0-1)\rho}\}\big|\\
&\lesssim _rP^{r(d_0-1)\rho}P^{C_1-r\rho/C_1}.
\end{split}
\end{equation}
for some constant $C_1=C_1(d)$ sufficiently large. To prove that \eqref{pro32} implies \eqref{pro30}, we prove that for $l=0,\ldots,d_0-1$ the number $N_{\rho,l}$ of solutions
\begin{equation}\label{pro41}
\begin{split}
&v^{(1)},\ldots,v^{(l)}\in B_{\Z^{r}}(P),\,v^{(l+1)},\ldots,v^{(d_0-1)}\in B_{\Z^{r}}(P^\rho),\\
&\sup_{j=1,\ldots,r}\|B_j(v^{(1)},\ldots,v^{(d_0-1)})\|\leq P^{-(d_0-l)+(d_0-1-l)\rho},
\end{split}
\end{equation}
satisfies
\begin{equation}\label{pro42}
N_{\rho,l}\lesssim _rP^{rl(1-\rho)}P^{r(d_0-1)\rho}P^{C_1-r\rho/C_1}.
\end{equation}
In the case $l=0$ this is equivalent to the assumption \eqref{pro32}. The claim \eqref{pro42} follows by induction over $l$, using Lemma \ref{Dav} at each step. The symmetry condition \eqref{pro40} is satisfied, in view of \eqref{pro401}-\eqref{pro400}. The case $l=d_0-1$ gives the desired conclusion \eqref{pro30}.
\medskip

{\bf{Step 4.}} For $j=1,\ldots,r$ and $(l_1,l_2)\in Y_d$ with $l_1+l_2=d_0$ let
\begin{equation}\label{pro50}
A_j^{l_1l_2}(v^{(1)},\ldots,v^{(d_0-1)})=d_0!\sum_{j_1,\ldots,j_{d_0-1}=1}^{r}\lambda^{l_1l_2}_{j_1\ldots j_{d_0-1}j}v^{(1)}_{j_1}\cdot\ldots\cdot v^{(d_0-1)}_{j_{d_0-1}},
\end{equation}
see \eqref{pro401}-\eqref{pro400}. For any $v^{(1)},\ldots,v^{(d_0-1)}$ fixed we think of $A_j^{l_1l_2}(v^{(1)},\ldots,v^{(d_0-1)})$ as a $r\times d_1$ matrix, where
\begin{equation*}
d_1=|Y_{d,d_0}|,\qquad Y_{d,d_0}=\{(l_1,l_2)\in Y_d:l_1+l_2=d_0\}.
\end{equation*}

We show now that
\begin{equation}\label{pro51}
\begin{split}
&\{v^{(1)},\ldots,v^{(d_0-1)}\in B_{\Z^{r}}(P^\rho):\sup_{j=1,\ldots,r}\|B_j(v^{(1)},\ldots,v^{(d_0-1)})\|\leq P^{-d_0+(d_0-1)\rho}\}\\
&\subseteq \{v^{(1)},\ldots,v^{(d_0-1)}\in B_{\Z^{r}}(P^\rho):\mathrm{rank}\big[A_j^{l_1l_2}(v^{(1)},\ldots,v^{(d_0-1)})\big]\leq d_1-1\},
\end{split}
\end{equation}
provided that the constant $C_0$ fixed in \eqref{pro1.5} is sufficiently large (depending only on $d$). To see this, as in the proof of Lemma 2.5 in \cite{Bi}, assume, for contradiction, that $\sup_{j=1,\ldots,r}\|B_j(v^{(1)},\ldots,v^{(d_0-1)})\|\leq P^{-d_0+(d_0-1)\rho}$ for some $v^{(1)},\ldots,v^{(d_0-1)}\in B_{\Z^{r}}(P^\rho)$ for which $\mathrm{rank}\big[A_j^{l_1l_2}(v^{(1)},\ldots,v^{(d_0-1)})\big]= d_1$. Notice that
\begin{equation*}
B_j(v^{(1)},\ldots,v^{(d_0-1)})=\sum_{l_1+l_2=d_0}Q\theta_{l_1l_2}A_j^{l_1l_2}(v^{(1)},\ldots,v^{(d_0-1)}).
\end{equation*}
We could then solve the linear system in the variables $Q\theta_{l_1l_2}$ to deduce that
\begin{equation*}
Q\theta_{l_1l_2}=\frac{m_{l_1l_2}}{n_{l_1l_2}}+\delta_{l_1l_2},\,\,m_{l_1l_2},n_{l_1l_2}\in\mathbb{Z},\,\,1\leq n_{l_1l_2}\lesssim_r P^{d_1(d_0-1)\rho},\,\,|\delta_{l_1l_2}|\lesssim_rP^{-d_0+d_1(d_0-1)\rho}
\end{equation*}
for any $(l_1,l_2)\in Y_{d,d_0}$. Recalling the bound $1\leq Q\leq P^{2d^2\delta_{d_0+1}}$ and the definition $\rho=(\delta_{d_0+1}\delta_{d_0})^{1/2}$, this is clearly in contradiction with \eqref{pro1}-\eqref{pro2} if $P$ is sufficiently large relative to $r$ and $C_0=\delta_{d_0}/\delta_{d_0+1}$ is sufficiently large relative to $d$.

Therefore, for \eqref{pro32} it suffices to prove that
\begin{equation}\label{pro52}
\begin{split}
\big|\{v^{(1)},\ldots,v^{(d_0-1)}\in B_{\Z^{r}}(P^\rho):\mathrm{rank}\big[A_j^{l_1l_2}(v^{(1)},\ldots,v^{(d_0-1)})\big]\leq d_1-1\}\big|&\\
\lesssim _rP^{r(d_0-1)\rho}P^{C_1-r\rho/C_1}&.
\end{split}
\end{equation}

Recall that (see \eqref{pro0.4})
\begin{equation}\label{pol}
D^0(m)_{l_1l_2}=\begin{cases}
\sum\limits_{1\leq j\leq r} m_j^{l_1}+R^0_{l_1l_2}(m)&\text{ if }(l_1,l_2)=(d_0,0),\\
\sum\limits_{1\leq j_1<j_2\leq r} m_{j_1}^{l_1} m_{j_2}^{l_2} +R^0_{l_1l_2}(m)&\text{ if }(l_1,l_2)\in Y_{d,d_0},\,l_2\geq 1,
\end{cases}
\end{equation}
where $R^0_{l_1l_2}$ are polynomials in $m$ of degree at most $d_0-1$. These polynomials give no contribution to the values of $A_j^{l_1l_2}(v^{(1)},\ldots,v^{(d_0-1)})$. Using the definitions, it follows that for fixed $1\leq j\leq r$
\begin{equation}\label{multilin2.0}
A_j^{l_1l_2}(v^{(1)},\ldots,v^{(d_0-1)})=
d_0\sum\limits_{\si}v_j^{(\si_1)}\ldots v_j^{(\si_{d_0-1})}
\end{equation}
if $(l_1,l_2)=(d_0,0)$, and
\begin{equation}\label{multilin2.1}
\begin{split}
A_j^{l_1l_2}(v^{(1)},\ldots,v^{(d_0-1)})
&=l_1\sum\limits_\si \sum\limits_{j<k\leq r} v_j^{(\si_1)}\ldots v_j^{(\si_{l_1-1})}\,v_k^{(\si_{l_1})}\ldots v_k^{(\si_{d_0-1})}\\
&+l_2\sum\limits_\si\sum\limits_{1\leq k<j} v_k^{(\si_1)}\ldots v_k^{(\si_{l_1})}\,v_j^{(\si_{l_1+1})}\ldots v_j^{(\si_{d_0-1})},
\end{split}
\end{equation}
if $(l_1,l_2)\in Y_{d,d_0},\,l_2\geq 1$. Here $\si=(\si_1,\ldots ,\si_{d_0-1})$ runs through all the permutations of the set $\{1,2,\ldots,d_0-1\}$.
\medskip

{\bf{Step 5.}} We examine now the set in the left-hand side of \eqref{pro52}. Since the matrix coefficients $A_j^{l_1l_2}(v^{(1)},\ldots,v^{(d_0-1)})$ are integers and of size $\lesssim P^{(d_0-1)\rho}$, it is easy to see from Cramer's rule that if $\mathrm{rank}\big[A_j^{l_1l_2}(v^{(1)},\ldots,v^{(d_0-1)})\big]\leq d_1-1$ for some $v^{(1)},\ldots,v^{(d_0-1)}\in B_{\Z^{r}}(P^\rho)$ then there exists a set of integers $b_{l_1l_2}$ not all zero of size $|b_{l_1l_2}|\lesssim P^{C_2\rho}$ (with a constant $C_2$ depending only on $d$), such that
\begin{equation}\label{system1}
\sum_{(l_1,l_2)\in Y_{d,d_0}} b_{l_1 l_2}\,A_j^{l_1l_2}(v^{(1)},\ldots,v^{(d_0-1)})=0\qquad\text{for all}\ \ \ \ 1\leq j\leq r.
\end{equation}

For a given permutation $\si=(\si_1,\ldots,\si_{d_0-1})$ and a given pair $(l_1,l_2)\in Y_{d,d_0}$ such that $l_2\geq 1$, define
\begin{equation*}
T^\si_{l_1l_2}=\sum\limits_{k=1}^r v_k^{(\si_{l_1})}\ldots v_k^{(\si_{d_0-1})}.
\end{equation*}
We define, compare with \ref{multilin2.1},
\begin{equation}\label{multilin2.2}
\begin{split}
\widetilde{A}_j^{l_1l_2}(v^{(1)},\ldots,v^{(d_0-1)})
&=l_2\sum\limits_\si\sum\limits_{1\leq k<j} v_k^{(\si_1)}\ldots v_k^{(\si_{l_1})}\,v_j^{(\si_{l_1+1})}\ldots v_j^{(\si_{d_0-1})}\\
&-l_1\sum\limits_\si \sum\limits_{1\leq k\leq j} v_j^{(\si_1)}\ldots v_j^{(\si_{l_1-1})}\,v_k^{(\si_{l_1})}\ldots v_k^{(\si_{d_0-1})}\\
&+l_1\sum\limits_\si T^\si_{l_1l_2}\,v_j^{(\si_1)}\ldots v_j^{(\si_{l_1-1})}.
\end{split}
\end{equation}
The advantage of formula \ref{multilin2.2} is that for any fixed values of the parameters $T^\si_{l_1l_2}$, the quantities $\widetilde{A}_j^{l_1l_2}(v^{(1)},\ldots,v^{(d_0-1)})$ depend only on the variables $v^{(m)}_k$ for $1\leq m\leq d_0-1$ and $1\leq k\leq j$. We define also, compare with \eqref{multilin2.0},
\begin{equation*}
A_j^{d_00}(v^{(1)},\ldots,v^{(d_0-1)})=A_j^{d_00}(v^{(1)},\ldots,v^{(d_0-1)})=d_0\sum\limits_{\si}v_j^{(\si_1)}\ldots v_j^{(\si_{d_0-1})}.
\end{equation*}

Using these definitions and \eqref{system1}, we conclude that if $(v^{(1)},\ldots,v^{(d_0-1)})$ is an element of the set in the left-hand side of \eqref{pro52} then there are integers $b_{l_1l_2}$ (not all zero) and $T^\si_{l_1l_2}$ in $[-P,P]$ such that
\begin{equation*}
\sum_{(l_1,l_2)\in Y_{d,d_0}} b_{l_1 l_2}\,\widetilde{A}_j^{l_1l_2}(v^{(1)},\ldots,v^{(d_0-1)})=0\qquad\text{for all}\ \ \ \ 1\leq j\leq r.
\end{equation*}
Therefore, for \eqref{pro52} it suffices to prove that for any integers $b_{l_1l_2}$ (not all zero) and $T^\si_{l_1l_2}$ in $[-P,P]$
\begin{equation}\label{pro70}
\begin{split}
&\big|\{v^{(1)},\ldots,v^{(d_0-1)}\in B_{\Z^{r}}(P^\rho):\\
&\sum_{(l_1,l_2)\in Y_{d,d_0}} b_{l_1 l_2}\,\widetilde{A}_j^{l_1l_2}(v^{(1)},\ldots,v^{(d_0-1)})=0\,\text{for all }\,1\leq j\leq r\}\big|\lesssim _rP^{r(d_0-1)\rho}P^{-r\rho/C_1}.
\end{split}
\end{equation}
\medskip

{\bf{Step 6.}} Finally, we prove \eqref{pro70} using the simple Lemma \ref{roots} below. Let $1\leq j\leq r$ be a given even integer. For any given choice of the parameters $b_{l_1l_2}$ (not all zero), $T^\si_{l_1l_2}$ and for any given values of the variables $v^{(h)}_k,\ 1\leq h\leq d_0-1, 1\leq k\leq j-2$ we claim that that
\begin{equation}\label{key}
\sum_{(l_1,l_2)\in Y_{d,d_0}} b_{l_1 l_2}\,\widetilde{A}_j^{l_1l_2}(v^{(1)},\ldots,v^{(d_0-1)})
\end{equation}
is not identically zero as a polynomial in the variables $v^{(1)}_{j-1},\ v^{(1)}_j,\ldots,v^{(d_0-1)}_{j-1},\ v^{(d_0-1)}_j$. Indeed, if $b_{l_1l_2}\neq 0$ for a pair $(l_1,l_2)\in Y_{d,d_0}\setminus \{(d_0,0)\}$, then, for any permutation $\si$, the expression \ref{key} contains the term
\begin{equation*}
b_{l_1l_2} l_2v_{j-1}^{(\si_1)}\ldots v_{j-1}^{(\si_{l_1})}\,v_j^{(\si_{l_1+1})}\ldots v_j^{(\si_{d_0-1})}.
\end{equation*}
 If, on the other hand, $b_{d_00}\neq 0$ for $l_1=d_0-1,\ l_2=0$ but $b_{l_1l_2}=0$ for all pairs $(l_1,l_2)\in Y_{d,d_0}\setminus \{(d_0,0)\}$, then the expression \ref{key} takes the form
\begin{equation*}
b_{d_00}d_0\sum\limits_{\si}v_j^{(\si_1)}\ldots v_j^{(\si_{d_0-1})}
\end{equation*}
which is not identically zero.

Therefore we may apply estimate \ref{bound} repeatedly for $j=2,4,\ldots$. It follows that the number of solutions $(v^{(1)},\ldots,v^{(d_0-1)})\in B_{\Z^{(d_0-1)r}}(P^\rho)$ of the system of equations in \eqref{pro70} is $\lesssim P^{r(d_0-1)\rho-r\rho/2}$, as desired.
\end{proof}

\begin{lemma}\label{roots}
Assume that $P=P(x_1,\ldots,x_s)$ is a polynomial of degree $d$ in $s$ variables which is not identically $0$, and $A\subseteq\mathbb{R}$. Then
\begin{equation}\label{bound}
\big|\{(x_1,\ldots,x_s)\in A^s:\ P(x_1,\ldots,x_s)=0\}\big|\leq d|A|^{s-1}.
\end{equation}
\end{lemma}

\begin{proof}[Proof of Lemma \ref{roots}] The statement is immediate when $s=1$ or $d=1$. We proceed by induction. Without loss of generality assume that
\begin{equation*}
P(x_1,\ldots,x_s)=Q(x_2,\ldots,x_s) x_1^{d_1}+R(x_1,\ldots,x_s)
\end{equation*}
where $Q(x_2,\ldots,x_s)$ is a polynomial of degree at most $d-d_1$ not identically zero. If $Q(x_2,\ldots,x_s)\neq 0$ then there are at most $d_1$ values of $x_1$ for which $P(x_1,x_2,\ldots,x_s)=0$. Thus, by induction, the left-hand side of \ref{bound} is estimated by
\begin{equation*}
d_1|A|^{s-1}+(d-d_1)|A|^{s-2}|A|=d|A|^{s-1},
\end{equation*}
as desired.
\end{proof}

We conclude this section with an estimate on an oscillatory integral. We think of $D,\widetilde{D}$ as functions defined on $\mathbb{R}^r\times\mathbb{R}^r$ taking values in $\mathbb{R}^{|Y_d|}$, given by \eqref{pro0.4} and \eqref{pro0.5}.

\begin{lemma}\label{intest}
Assume $\Phi:\mathbb{R}^r\times\mathbb{R}^r\to\mathbb{R}$ satisfies
\begin{equation}\label{cro1}
|\partial_{x_1}^{\sigma_1}\ldots\partial_{x_r}^{\sigma_r}\partial_{y_1}^{\va_1}\ldots\partial_{y_r}^{\va_r}\Phi(x,y)|\leq\mathbf{1}_{B_{\mathbb{R}^r}(1)}(x)\mathbf{1}_{B_{\mathbb{R}^r}(1)}(y)
\end{equation}
for any $\sigma_1,\ldots,\sigma_r,\va_1,\ldots,\va_r\in\{0,1\}$, where $B_{\mathbb{R}^m}(C)=\{x\in\mathbb{R}^m:|x|\leq C\}$. Then there is a constant $\overline{C}=\overline{C}(d)$ sufficiently large such that for any $\beta\in\mathbb{R}^{|Y_d|}$,
\begin{equation}\label{cro2}
\Big|\int_{\mathbb{R}^r\times\mathbb{R}^r}\Phi(x,y)e^{-2\pi iD(x,y)\cdot \beta}\,dxdy\Big|+\Big|\int_{\mathbb{R}^r\times\mathbb{R}^r}\Phi(x,y)e^{-2\pi i\widetilde{D}(x,y)\cdot \beta}dxdy\Big|\lesssim _r (1+|\beta|)^{\overline{C}-r/\overline{C}}.
\end{equation}
\end{lemma}

\begin{proof}[Proof of Lemma \ref{intest}] We will only prove the estimate on the first term in the left-hand side of \eqref{cro2}, using Proposition \ref{minarcs}. Let $C_0$, $C_1$ are suitably large fixed constants (depending on the constant $\overline{C}$ in Proposition \ref{minarcs}), assume $|\be|\geq C_0$, and choose $\eps=C_1/r$ in Proposition \ref{minarcs}. Assume that
\begin{equation*}
(n_1,n_2)\in Y_d,\qquad |\beta_{n_1n_2}|=\sup_{(l_1,l_2)\in Y_d}|\beta_{l_1l_2}|.
\end{equation*}
Let $P$ be a positive number, so that $P\approx |\be|^{1/\eps}$ and $q:=P^{n_1+n_2} |\be_{n_1n_2}|^{-1}$ is an integer.

By rescaling one may write
\begin{equation*}
I_D(\be):=\int_{\mathbb{R}^r\times\mathbb{R}^r} \Phi(x,y) e^{-2\pi i D(x,y)\cdot\be}\,dxdy = P^{-2r}\int_{\mathbb{R}^r\times\mathbb{R}^r} \Phi(\frac{x}{P},\frac{y}{P}) e^{-2\pi i D(x,y)\cdot\th}\,dx dy
\end{equation*}
where
\begin{equation*}
\theta_{l_1l_2}=P^{-(l_1+l_2)}\beta_{l_1l_2},\qquad (l_1,l_2)\in Y_d.
\end{equation*}
Note that $\th_{n_1n_2}=\pm 1/q$ with $q\approx P^{n_1+n_2-\eps}$. Therefore, by Proposition \ref{minarcs}, one has the estimate
\begin{equation*}
P^{-2r}\,S_{P,r}(\th):=P^{-2r}\,\sum_{(n,m)\in\mathbb{Z}^{r}\times\mathbb{Z}^r} \Phi(n/P,m/P) e^{-2\pi i D(n,m)\cdot\th} \lesssim_r P^{\bar{C}-r\eps/\bar{C}}\lesssim_r P^{-1}.
\end{equation*}

On the other hand writing $x=n+s$, $y=m+t$ with $m,n\in\Z^r$ and $s,t\in [0,1)^r$ it is easy to see that
\begin{equation*}
|I_D(\be)-P^{-2r}\,S_{P,r}(\th)|\lesssim \sum_{(l_1,l_2)\in Y_d}|\th_{l_1l_2}| P^{l_1+l_2-1}+P^{-1} \lesssim P^{-1/2}.
\end{equation*}
This gives the estimate $|I_D(\be)|\lesssim_r |\be|^{-r/2C_1}$ for $|\be|\geq C_0$ and the lemma follows.
\end{proof}

\section{An almost orthogonality lemma}\label{orthog}

We assume that $H$ is a Hilbert space, $S_m\in\mathcal{L}(H)$, $m=1,\ldots,K$, are self-adjoint operators, and
\begin{equation}\label{ma1}
\|S_m\|\leq 1,\quad m=1,\ldots,K.
\end{equation}
Let
\begin{equation*}
I=\{0,1\},\qquad S_{m,0}=S_m,\qquad S_{m,1}=0.
\end{equation*}
For any dyadic integer $p$ we define
\begin{equation}\label{ma2}
B_p=\sup_{i_1,\ldots,i_K\in I}\|S_{1,i_1}^p+S_{2,i_2}^p+\ldots+S_{K,i_K}^p\|.
\end{equation}
and, for any $m=1,\ldots,K-1$ and dyadic integer $p$
\begin{equation}\label{ma2.1}
\gamma_{m,p}=\sup_{i_m,\ldots,i_K\in I}\|S_{m,i_m}(S_{m+1,i_{m+1}}^p+\ldots+S_{K,i_K}^p)\|.
\end{equation}
We start with a lemma:

\begin{lemma}\label{mainlemma}
Assume that $S_{m,i}$, $B_p$, $\gamma_{mp}$ are as above and that there are constants $\delta_0>0$, $A\geq 1$ and a dyadic integer $p_0$ such that
\begin{equation}\label{ma3}
\gamma_{m,p_0}\leq A2^{-\delta_0m}(B_{p_0}+1)\quad\text{ for }m=1,\ldots,K-1.
\end{equation}
Then
\begin{equation}\label{ma8}
\begin{split}
&B_{1}\leq C(\delta_0,A,p_0),\\
&\gamma_{m,1}\leq C(\delta_0,A,p_0)2^{-\delta'_0m},\qquad m=1,\ldots,K-1,
\end{split}
\end{equation}
for some constants $C=C(\delta_0,A,p_0)\in[1,\infty)$ and $\delta'_0=\delta'_0(\delta_0,A,p_0)>0$.
\end{lemma}

\begin{proof} We prove the lemma in two steps.

{\bf{Step 1.}} We show first that
\begin{equation}\label{show1}
B_{p_0}\leq C(\delta_0,A).
\end{equation}
Assume $p\geq p_0$ is a dyadic integer and fix $i_1,\ldots i_K\in I$ such that the supremum in \eqref{ma2} is attained. Then, using self-adjointness and \eqref{ma1}, we write
\begin{equation}\label{ma5}
\begin{split}
B_p^2&=\|(S_{1,i_1}^p+S_{2,i_2}^p+\ldots+S_{K,i_K}^p)^2\|\\
&\leq\|S_{1,i_1}^{2p}+\ldots+S_{K,i_K}^{2p}\|+2\sum_{m=1}^{K-1}\|S_{m,i_m}^p(S_{m+1,i_{m+1}}^p+\ldots+S_{K,i_K}^p)\|\\
&\leq B_{2p}+2\sum_{m=1}^{K-1}\gamma_{m,p}.
\end{split}
\end{equation}
We estimate also $\gamma_{m,2p}$. For any $j_m,\ldots,j_K\in I$
\begin{equation*}
\begin{split}
\|S_{m,j_m}(S_{m+1,j_{m+1}}^{2p}+\ldots+S_{K,j_K}^{2p})&\|\leq \|S_{m,j_m}(S_{m+1,j_{m+1}}^{p}+\ldots+S_{K,j_K}^{p})^2\|\\
&+2\sum_{m'=m+1}^{K-1}\|S_{m',j_{m'}}^p(S_{m'+1,j_{m'+1}}^p+\ldots+S_{K,j_K}^p)\|\\
&\leq B_p\gamma_{m,p}+2\sum_{m'=m+1}^{K-1}\gamma_{m',p},
\end{split}
\end{equation*}
using \eqref{ma1} and the identity
\begin{equation*}
\begin{split}
&S_{m+1,j_{m+1}}^{2p}+\ldots+S_{K,j_K}^{2p}=(S_{m+1,j_{m+1}}^{p}+\ldots+S_{K,j_K}^{p})^2\\
&-\sum_{m'=m+1}^{K-1}S_{m',j_{m'}}^p(S_{m'+1,j_{m'+1}}^p+\ldots+S_{K,j_K}^p)-(S_{m'+1,j_{m'+1}}^p+\ldots+S_{K,j_K}^p)S_{m',j_{m'}}^p.
\end{split}
\end{equation*}
Thus, for any $m=1,\ldots,K$ and any dyadic integer $p\geq p_0$
\begin{equation}\label{ma5.1}
\gamma_{m,2p}\leq B_p\gamma_{m,p}+2\sum_{m'=m+1}^{K-1}\gamma_{m',p}.
\end{equation}

We use now inequalities \eqref{ma3}, \eqref{ma5}, and \eqref{ma5.1} to prove \eqref{show1}. Let
\begin{equation*}
L=L(\delta_0)=\sum_{m=0}^\infty 2^{-\delta_0m}.
\end{equation*}
Let $p_1\geq p_0$ denote the smallest dyadic integer for which $B_{p_1}\leq (100LA)^{p_1}$. Such $p_1$ exists because $B_p\leq K$, using \eqref{ma1}. The bound \eqref{show1} follows if $p_1=p_0$. Otherwise we have, for any dyadic integer $p\in[p_0,p_1)$ and any $m=1,\ldots,K$
\begin{equation}\label{ma6}
\begin{split}
&B_p>(100LA)^p;\\
&B_p^2\leq B_{2p}+2\sum_{m=1}^{K-1}\gamma_{m,p};\\
&\gamma_{m,2p}\leq B_p\gamma_{m,p}+2\sum_{m'=m}^{K-1}\gamma_{m',p}.
\end{split}
\end{equation}
It follows from the second equation of \eqref{ma6} and \eqref{ma3} that
\begin{equation*}
B_{p_0}^2\leq B_{2p_0}+4ALB_{p_0}.
\end{equation*}
Using the first equation of \eqref{ma6} it follows that
\begin{equation*}
B_{p_0}^2\leq  2B_{2p_0}.
\end{equation*}
Using the third equation of \eqref{ma6} and \eqref{ma3} it follows that
\begin{equation*}
\gamma_{m,2p_0}\leq B_{p_0}2A2^{-\delta_0m}B_{p_0}+4ALB_{p_0}2^{-\delta_0m}\leq 2^{-\delta_0 m}B_{2p_0}(8A),
\end{equation*}
using $B_{p_0}\geq 2L$ and $B_{p_0}^2\leq 2B_{2p_0}$.

More generally, we  prove by induction that for any dyadic integer $p\in[p_0,p_1)$ and any $m=1,\ldots,K$
\begin{equation}\label{ma7}
B_p^2\leq 2B_{2p}\qquad\text{ and }\qquad \gamma_{m,2p}\leq 2^{-\delta_0m}B_{2p}(4A)^{2p}.
\end{equation}
This was already proved above for $p=p_0$. Assume $p\in[2p_0,p_1)$ is a dyadic integer. It follows from the second inequality in \eqref{ma6} and the induction hypothesis that
\begin{equation*}
B_p^2\leq B_{2p}+2L(4A)^pB_p.
\end{equation*}
Since $B_p>(100LA)^p$, this gives the first inequality in \eqref{ma7}. Using the third inequality in \eqref{ma6} and the induction hypothesis,
\begin{equation*}
\gamma_{m,2p}\leq B_p2^{-\delta_0m}B_{p}(4A)^{p}+2\cdot2^{-\delta_0m}LB_p(4A)^p\leq 2^{-\delta_0m}B_{2p}(4A)^{2p},
\end{equation*}
using $B_p^2\leq 2B_{2p}$ and $B_p\geq 2L$. By induction, this completes the proof of \eqref{ma7}.

Recall now that $B_{p_1}\leq (100LA)^{p_1}$. Thus, using only the first inequality in \eqref{ma7},
\begin{equation*}
\begin{split}
&B_{p_1/2}\leq 2^{1/2}(100LA)^{p_1/2}\\
&B_{p_1/4}\leq 2^{1/2}2^{1/4}(100LA)^{p_1/4}\\
&\ldots\\
&B_{p_1/2^l}\leq 2^{1/2}2^{1/4}\cdot \ldots \cdot 2^{1/2^l}(100LA)^{p_1/2^l}.
\end{split}
\end{equation*}
The bound \eqref{show1} follows by letting $2^l=p_1/p_0$.
\medskip

{\bf{Step 2.}} We prove now the bound \eqref{ma8}. It follows from \eqref{ma3} and \eqref{show1} that
\begin{equation}\label{ma20}
B_{p_0}\leq  A'\quad\text{ and }\quad\gamma_{m,p_0}\leq A'2^{-\delta_0m}\quad\text{ for }m=1,\ldots,K,
\end{equation}
for some constant $A'=A'(\delta_0,A)$. We would like to prove that, for some constant $A''=A''(A',\delta_0)$
\begin{equation}\label{ma21}
B_{p_0/2}\leq  A''\quad\text{ and }\quad\gamma_{m,p_0/2}\leq A''2^{-\delta_0m/4}\quad\text{ for }m=1,\ldots,K.
\end{equation}
We would then be able to prove \eqref{ma8} by repeating this step finitely many times.

We may assume $p_0\geq 2$ and look at $B_{p_0/2}$. Fix $i_1,\ldots,i_K\in I$ which attain the supremum in the definition of $B_{p_0/2}$ and write
\begin{equation}\label{ma10}
\begin{split}
B_{p_0/2}^2&=\|(S_{1,i_1}^{p_0/2}+\ldots+S_{K,i_K}^{p_0/2})^2\|\leq\|S_{1,i_1}^{p_0}+\ldots+S_{K,i_K}^{p_0}\|\\
&+2\sum_{m=1}^{K-1}\|S_{m,i_m}^{p_0/2}(S_{m+1,i_{m+1}}^{p_0/2}+\ldots+S_{K,i_K}^{p_0/2})\|\\
&\leq A'+2\sum_{m=1}^{K-1}\|S_{m,i_m}(S_{m+1,i_{m+1}}^{p_0/2}+\ldots+S_{K,i_K}^{p_0/2})\|,
\end{split}
\end{equation}
using \eqref{ma1}. Let
\begin{equation}\label{ma11}
Q=\sup_{m=1,\ldots, K-1}\sup_{j_m,\ldots,j_K\in I}2^{\delta_0 m/4}\|S_{m,j_m}(S_{m+1,j_{m+1}}^{p_0/2}+\ldots+S_{K,j_K}^{p_0/2})\|.
\end{equation}
Fix $m,j_m,\ldots,j_K$ such that the supremum in \eqref{ma11} is attained. Then we have
\begin{equation}\label{ma12}
\begin{split}
Q=2^{\delta_0m/4}\|S_{m,j_m}&(S_{m+1,j_{m+1}}^{p_0/2}+\ldots+S_{K,j_K}^{p_0/2})\|\leq 2^{\delta_0m/4}\sum_{m'=m+1}^{8m}\|S_{m,j_m}S_{m',j_{m'}}^{p_0/2}\|\\
&+2^{\delta_0m/4}\|S_{m,j_m}(S_{8m+1,j_{8m+1}}^{p_0/2}+\ldots+S_{K,j_K}^{p_0/2})\|.
\end{split}
\end{equation}
Now, using the second  inequality in \eqref{ma20} and the definition of $Q$ in \eqref{ma11}, \eqref{ma1}, selfadjointness, and the hypothesis $S_{m,0}=0$
\begin{equation*}
\|S_{m,j_m}S_{m',j_{m'}}^{p_0/2}\|^2\leq \|S_{m,j_m}S_{m',j_{m'}}^{p_0}S_{m,j_m}\|\leq A'2^{-\delta_0m},
\end{equation*}
and
\begin{equation*}
\begin{split}
\|S_{m,j_m}&(S_{8m+1,j_{8m+1}}^{p_0/2}+\ldots+S_{K,j_K}^{p_0/2})\|^2\leq \|S_{m,j_m}(S_{8m+1,j_{8m+1}}^{p_0}+\ldots+S_{K,j_K}^{p_0})\|\\
&+2\sum_{m'\geq 8m}\|S_{m',j_{m'}}^{p_0/2}(S_{m'+1,j_{m'+1}}^{p_0/2}+\ldots+S_{K,j_K}^{p_0/2})\|\\
&\leq A'2^{-\delta_0m}+2\sum_{m'\geq 8m}Q2^{-\delta_0m'}\leq 2^{-\delta_0 m}(A'+2LQ).
\end{split}
\end{equation*}
Therefore, it follows from \eqref{ma12} and the last two inequalities that
\begin{equation*}
Q\leq 2^{\delta_0m/4}2^{-\delta_0m/2}\sqrt{A'}(7m)+2^{\delta_0m/4}2^{-\delta_0m/2}\sqrt{A'+2LQ}\leq C_{\delta_0}\sqrt{A'+2LQ}.
\end{equation*}
It follows that $Q\leq C(\delta_0,A')$. In view of the definition \eqref{ma11}, this proves the second inequality in \eqref{ma21}. The first inequality in \eqref{ma21} follows from \eqref{ma10}. This completes the proof  of the lemma.
\end{proof}

We will need a version of this lemma for non-selfadjoint operators.

\begin{lemma}\label{prop1}
Assume that $H$ is a Hilbert space, $S_m\in\mathcal{L}(H)$, $m=1,\ldots,K$, and
\begin{equation}\label{na1}
\|S_m\|\leq 1,\quad m=1,\ldots,K.
\end{equation}
Let
\begin{equation*}
I=\{0,1\},\qquad S_{m,0}=S_m,\qquad S_{m,1}=0.
\end{equation*}
For any dyadic integer $p$ we define
\begin{equation}\label{na2}
\begin{split}
&D_p=\sup_{i_1,\ldots,i_K\in I}\|(S_{1,i_1}S^\ast_{1,i_1})^p+\ldots+(S_{K,i_K}S^\ast_{K,i_K})^p\|,\\
&\widetilde{D}_p=\sup_{i_1,\ldots,i_K\in I}\|(S^\ast_{1,i_1}S_{1,i_1})^p+\ldots+(S^\ast_{K,i_K}S_{K,i_K})^p\|.
\end{split}
\end{equation}
For any $m=1,\ldots,K-1$ and dyadic integer $p$ we define
\begin{equation}\label{na2.1}
\begin{split}
&\mu_{m,p}=\sup_{i_m,\ldots,i_K\in I}\|(S_{m,i_m}S^{\ast}_{m,i_m})[(S_{m+1,i_{m+1}}S^\ast_{m+1,i_{m+1}})^p+\ldots+(S_{K,i_K}S^\ast_{K,i_K})^p]\|,\\
&\widetilde{\mu}_{m,p}=\sup_{i_m,\ldots,i_K\in I}\|(S^\ast_{m,i_m}S_{m,i_m})[(S^\ast_{m+1,i_{m+1}}S_{m+1,i_{m+1}})^p+\ldots+(S^\ast_{K,i_K}S_{K,i_K})^p]\|.
\end{split}
\end{equation}
Assume that
\begin{equation}\label{na2.2}
\mu_{m,p_0}\leq A2^{-\delta_0m}(D_{p_0}+1)\text{ and }\widetilde{\mu}_{m,p_0}\leq A2^{-\delta_0m}(\widetilde{D}_{p_0}+1),\qquad m=1,\ldots,K-1,
\end{equation}
for some dyadic integer $p_0$ and some numbers $A\geq 1$ and $\delta_0>0$. Then
\begin{equation}\label{na2.3}
\|S_1+\ldots+S_K\|\leq C(\delta_0,A,p_0).
\end{equation}
\end{lemma}

\begin{remark}\label{prop2}
A simplified version of the lemma, which is used in the paper, is the following: assume that $H$ is a Hilbert space, $S_m\in\mathcal{L}(H)$, $m=1,\ldots,K$, and let $S_{m,0}=S_m$, $S_{m,1}=0$. Assume that, for all $m=1,\ldots,K$,
\begin{equation}\label{na99}
\begin{split}
&\sup_{m\in\{1,\ldots,K\}}\|S_m\|\leq 1,\\
&\sup_{i_m,\ldots,i_K\in I}\|S^{\ast}_{m,i_m}[(S_{m+1,i_{m+1}}S^\ast_{m+1,i_{m+1}})^{p_0}+\ldots+(S_{K,i_K}S^\ast_{K,i_K})^{p_0}]\|\leq A2^{-\delta_0m},\\
&\sup_{i_m,\ldots,i_K\in I}\|S_{m,i_m}[(S^\ast_{m+1,i_{m+1}}S_{m+1,i_{m+1}})^{p_0}+\ldots+(S^\ast_{K,i_K}S_{K,i_K})^{p_0}]\|\leq A2^{-\delta_0m}.
\end{split}
\end{equation}
Then
\begin{equation*}
\|S_1+\ldots+S_K\|\leq C(\delta_0,A,p_0).
\end{equation*}
\end{remark}

\begin{proof} [Proof of Lemma \ref{prop1}] We apply Lemma \ref{mainlemma} to the operators $S_mS^\ast_m$ and $S^\ast_mS_m$. It follows that there are constants $\overline{A}\geq 1$ and $\overline{\delta}>0$ depending only on $\delta_0, A, P_0$ such that
\begin{equation}\label{na5}
D_1+\widetilde{D}_1\leq\overline{A},\qquad\mu_{m,1}+\widetilde{\mu}_{m,1}\leq \overline{A}2^{-\overline{\delta}m},\qquad m=1,\ldots,K.
\end{equation}

For any $m=1,\ldots,K-1$ let
\begin{equation*}
\begin{split}
&\nu_{m}=\sup_{i_m,\ldots,i_K\in I}\|S^{\ast}_{m,i_m}[(S_{m+1,i_{m+1}}S^\ast_{m+1,i_{m+1}})+\ldots+(S_{K,i_K}S^\ast_{K,i_K})]\|,\\
&\widetilde{\nu}_{m}=\sup_{i_m,\ldots,i_K\in I}\|S_{m,i_m}[(S^\ast_{m+1,i_{m+1}}S_{m+1,i_{m+1}})+\ldots+(S^\ast_{K,i_K}S_{K,i_K})]\|.
\end{split}
\end{equation*}
Clearly, for any $m=1,\ldots,K-1$
\begin{equation*}
\nu_m^2\leq D_1\mu_{m,1},\qquad \widetilde{\nu}_m^2\leq\widetilde{D}_1\widetilde{\mu}_{m,1}.
\end{equation*}
Therefore, using \eqref{na5},
\begin{equation}\label{na6}
\nu_m+\widetilde{\nu}_m\leq 2\overline{A}2^{-\overline{\delta}m/2},\qquad m=1,\ldots,K.
\end{equation}

Clearly
\begin{equation*}
\|S_1+\ldots+S_K\|^2\leq \|S_1S_1^\ast+\ldots+S_KS_k^\ast\|+2\sum_{m=1}^{K-1}\|S_m(S_{m+1}^\ast+\ldots+S_K^\ast)\|.
\end{equation*}
Since $D_1\leq\overline{A}$, for \eqref{na2.3} it suffices to prove that
\begin{equation}\label{na7}
\|S_m(S_{m+1}^\ast+\ldots+S_K^\ast)\|\leq A'2^{-\overline{\delta}m/8},\qquad m=1,\ldots,K-1.
\end{equation}

Let
\begin{equation*}
\begin{split}
&Q=\sup_{m=1,\ldots,K-1}\sup_{i_m,\ldots,i_K\in I}2^{\overline{\delta}m/8}\|S_{m,i_m}(S_{m+1,i_{m+1}}^\ast+\ldots+S_{K,i_K}^\ast)\|,\\
&\widetilde{Q}=\sup_{m=1,\ldots,K-1}\sup_{i_m,\ldots,i_K\in I}2^{\overline{\delta}m/8}\|S^\ast_{m,i_m}(S_{m+1,i_{m+1}}+\ldots+S_{K,i_K})\|.
\end{split}
\end{equation*}
Fix $m,i_m,\ldots,i_K$ such that the supremum in the definition of $Q$ is attained. Then
\begin{equation}\label{na8}
Q\leq 2^{\overline{\delta}m/8}\sum_{m'=m+1}^{8m}\|S_{m,i_m}S^\ast_{m',i_{m'}}\|+2^{\overline{\delta}m/8}\|S_{m,i_m}(S_{8m+1,i_{8m+1}}^\ast+\ldots+S_{K,i_K}^\ast)\|.
\end{equation}
For any $m'\in[m+1,8m]\cap\Z$ we have, using \eqref{na6},
\begin{equation*}
\|S_{m,i_m}S^\ast_{m',i_{m'}}\|\leq \|S_{m,i_m}S^\ast_{m',i_{m'}}S_{m',i_{m'}}\|^{1/2}\leq \widetilde{\nu}_m^{1/2}\leq 2\overline{A}2^{-\overline{\delta}m/4}.
\end{equation*}
Using $\|S_m\|\leq 1$ and the definitions, it follows that
\begin{equation*}
\begin{split}
\|S_{m,i_m}&(S_{8m+1,i_{8m+1}}^\ast+\ldots+S_{K,i_K}^\ast)\|^2\\
&\leq \|S_{m,i_m}(S_{8m+1,i_{8m+1}}^\ast+\ldots+S_{K,i_K}^\ast)(S_{8m+1,i_{8m+1}}+\ldots+S_{K,i_K})\|\\
&\leq \widetilde{\nu}_m+2\sum_{m''=8m+1}^K\|S_{m'',i_{m''}}^\ast(S_{m''+1,i_{m''+1}}+\ldots+S_{K,i_K})\|\\
&\leq \widetilde{\nu}_m+2\sum_{m''=8m+1}^K\widetilde{Q}2^{-\overline{\delta}m''/8}.
\end{split}
\end{equation*}
Therefore, using \eqref{na6} and \eqref{na8},
\begin{equation*}
Q\leq C(\overline{\delta},\overline{A})(1+\widetilde{Q}^{1/2}).
\end{equation*}
A similar argument shows that
\begin{equation*}
\widetilde{Q}\leq C(\overline{\delta},\overline{A})(1+Q^{1/2}),
\end{equation*}
and the desired bound \eqref{na7} follows.
\end{proof}

\end{document}